\documentclass[a4paper,12pt]{article}
\usepackage{amsfonts}
\usepackage{amsmath,amssymb,amsthm,amsopn,exscale, wasysym}
\usepackage[T1]{fontenc}

\DeclareMathOperator{\im}{im}

\newtheorem{thm}{Theorem}
\newtheorem{lem}[thm]{Lemma}

\newtheorem{exams}[thm]{Examples}
\newtheorem{exam}[thm]{Example}
\newtheorem{ew}[thm]{Warning-Example}
\newtheorem{fe}[thm]{Fundamental Example}
\newtheorem{f}[thm]{Fact}
\newtheorem{fs}[thm]{Facts}
\newtheorem{rem}[thm]{Remark}
\newtheorem{toprem}[thm]{Topological Remark}
\newtheorem{cor}[thm]{Corollary}

\newcommand{\To}{\stackrel{\circ}{\mathcal{T}}}
\newcommand{\N}{\ensuremath{\mathbb{N}}}
\newcommand{\R}{\ensuremath{\mathbb{R}}}
\newcommand{\mb}{\mathbb}
\newcommand{\mc}{\mathcal}
\newcommand{\spa}{\mathbf{Space}}
\newcommand{\ads}{\mathbf{ADS}}
\newcommand{\ds}{\mathbf{DS}}
\newcommand{\lds}{\mathbf{LDS}}
\newcommand{\rplds}{\mathbf{RPLDS}}
\newcommand{\rplss}{\mathbf{RPLSS}}
\newcommand{\wds}{\mathbf{WDS}}
\newcommand{\gts}{\mathbf{GTS}}

\begin{document}
\begin{center}

\Large
\textbf{O-minimal homotopy  and generalized (co)homology}

\vspace{3mm}
\large \textit{} \textsc{Artur Pi\k{e}kosz}
\end{center}

\begin{abstract}
This article explains and extends  semialgebraic homotopy theory (developed by H. Delfs and M. Knebusch) to o-minimal homotopy theory (over a field). The homotopy category of definable CW-complexes is equivalent to 
the homotopy category of topological CW-complexes (with continuous mappings). If the theory of the o-minimal expansion of a field is bounded, then these categories are equivalent to the homotopy category of weakly definable spaces. Similar facts hold for decreasing systems of spaces.
As a result, generalized homology and cohomology theories on pointed weak polytopes uniquely  correspond (up to an isomorphism) to the known topological generalized homology and cohomology theories on pointed CW-complexes. 
\end{abstract}
{\small 2000 \textit{Mathematics Subject Classification:} 	03C64, 55N20, 55Q05. 
}

\section{Introduction}
In the 1980's, H. Delfs, M. Knebusch and others developed ``semialgebraic topology'' in locally semialgebraic and weakly semialgebraic spaces (see \cite{DK2,DK5,DK6,LSS,WSS}).
In the survey paper \cite{K91}, M. Knebusch suggested that this theory may be generalized to the o-minimal context. This programme was partially undertaken first by A. Woerheide, who constructed the o-minimal singular homology theory in \cite{W}, and later by M. Edmundo, who developed and applied the singular homology and cohomology theories over o-minimal structures (see for example \cite{E}). For homotopy theory, A. Berarducci and M. Otero worked with the o-minimal fundamental group and transfer methods in o-minimal  geometry (\cite{BO,BO2}). During the period this paper was written, several authors wrote about different types of homology and cohomology (see \cite{EJP,EP}, for example). 

Still the semialgebraic homotopy theory contained in \cite{LSS} and \cite{WSS} was not extended to the case of spaces over o-minimal expansions of fields.
For the question why, the author may only guess that people in the field wanted to avoid generalized topology.
(Notice the failure of E.Baro and M. Otero \cite{ldh} to give precise definitions and to present the theory clearly, see below.)

The aim of extending a whole theory, not a single theorem or even tens or hundreds of facts, may be sometimes achieved by carefull choice of the definitions and explaining the differences that  appear.
This can be done in the case of the semialgebraic homotopy theory of H. Delfs and M. Knebusch.

First, the spaces of our interest (with their morphisms) over each 
of the considered structures form several categories that are best described as full subcategories of some ambient category. The choice of a good ambient category is very important. In \cite{LSS} the tusk was done using sheaf theory, but M. Knebusch in \cite{WSS} has already simplified the definitions by using what is called ``function sheaves'' (involving a simple set-theoretic definition). Notice that the usual sheaf theory is not necessary to understand locally semialgebraic spaces.
 Thus the extension of the theory should be done through extension of the basic definitions from \cite{WSS}. Another argument for this is the fact that locally definable spaces do not suffice, we need to speak about weakly definable spaces to get a satisfactory homotopy theory.

Second,  some proofs of \cite{LSS} need modification. The mapping spaces from III.3 are specific for the semialgebraic case. This is modified in the present paper. Moreover,  Lemma II.4.3 from \cite{LSS} (and related facts) need to be modified since one needs to add the third Comparison Theorem (the o-minimal \textit{expansion case}). This was done in \cite{BaOt} by the use of ``normal triangulations'' from \cite{Ba} (the problem appears on the definable sets level).

And third, we need to distinguish between theories that are \textit{bounded} (definition in the present paper) and  other that are not bounded. The theory RCF itself is bounded, and some proofs of \cite{WSS} (related to IV.9-10) do not work in the general setting of an o-minimal expansion of a (real closed) field. The question arises if the corresponding facts  are true.

After considering these remarks, one can see that the two volumes \cite{LSS} and \cite{WSS} are a source of thousands of facts and their proofs about locally definable spaces and weakly definable spaces. It is usually done just by changing the word ``semialgebraic'' into the word ``definable''.
The intention of the author of the present paper is not to re-write about 600 pages with this simple change, but to give enough understanding of the theory to the reader. Some examples and facts from \cite{LSS} and \cite{WSS} are restated to make this understanding easy. 
(The above remarks apply to so-called ``geometric'' theory. The so-called ``abstract'' theory, contained in Appendix A of \cite{LSS}, is not considered in the present paper.)

It is convenient to understand that the semialgebraic homotopy theory of H. Delfs and M. Knebusch is basically the usual homotopy theory re-done in the presence of the generalized topology. The constructions of homotopy theory may be carried out in the semialgebraic context. Thus it is not surprising that these constructions may be also done in the context of o-minimal expansions of fields.
The use of the generalized topology may be extended far beyond the above context (see \cite{ap2} for details).  

The author considers the main result of this paper to be the following:
the semialgebraic homotopy theory of H. Delfs and M. Knebusch is now explained and extended to the o-minimal homotopy theory (over a field).
The extension part  includes the Comparison Theorems (especially Theorems \ref{3compld} and \ref{3compwd}), a definable version of the Whitehead theorem (Theorem \ref{cwwhitehead}) and equivalence of the homotopy categories (Corollaries \ref{concld},\ref{equi},\ref{concwd},\ref{conc}).
Majority of the examples and Theorems \ref{realline} and \ref{nonreal} contribute to the explanation part.
Of independent interest are: a characterization of real analytic manifolds as locally definable manifolds (Theorem \ref{manifold}) and the definable version of the Bertini-Lefschetz Theorem (Theorem \ref{bertini}).  
 
As a result of the homotopy approach, deeper than the homology and cohomology one, we get the generalized homology and cohomology theories (including the standard singular theories) for so called pointed weak polytopes, and these theories appear, if $T$ is bounded, to be ``the same'' as their topological counterparts.  

The categories of locally and weakly definable spaces over o-minimal expansions of real closed fields, introduced here,  with their subspaces (locally definable subsets and weakly definable subsets) are far generalizations of \textit{analytic-geometric categories} of van den Dries and Miller (\cite{DM}). In particular paracompact locally definable manifolds are generalizations of both definable manifolds  over o-minimal expansions of fields and real analytic manifolds.

 For basic properties of o-minimal structures, see the book \cite{Dries} and the survey paper \cite{DM}.
\textbf{Assume that $R$ is an o-minimal expansion of  a real closed field.}

\section{Spaces over o-minimal structures}

As o-minimal structures have natural topology, it is quite
natural that algebraic topology for such structures
should be developed. 
(This paper deals only with the case of o-minimal expansions of fields.) 
Unfortunately, 
there are obstacles to the above when one is doing traditional topology:
 if $R$ is not (an expansion of) the (ordered) field of real numbers $\mathbb{R}$,
then $R$ is not locally compact and is totally disconnected.
Moreover, even for $\mb{R}$, not every family of open definable sets has a definable union, and continuous definable functions do not form a sheaf.

A good idea to overcome that in the case of o-minimal pure (ordered) fields
was given by H. Delfs and M. Knebusch in \cite{LSS}: it is the concept of a generalized topological space. This idea serves well also in our setting.

\vspace{5mm}
A \textbf{generalized topological space} is a set $M$ together with a 
family of subsets $\stackrel{\circ}{\mathcal{T}} (M)$ of $M$,
called \textbf{open sets}, and a family of open families $\mathrm{Cov}_M $,
 called
 \textbf{admissible (open coverings)},
such that:
\begin{description}
\item[(A1)] $\emptyset , M\in \To (M)$ 
(the empty set and the whole space are open),

\item[(A2)] if $U_1 ,U_2 \in \To (M)$ then $U_1 \cup U_2 ,U_1
 \cap U_2 \in \To (M)$
(finite unions and finite intersections of open sets are open),

\item[(A3)] if $\{ U_i \}_{i\in I} \subseteq \To (M)$ and $I$ is finite, then
$\{ U_i \}_{i\in I} \in \mathrm{Cov}_M $
(finite families of open sets are admissible),

\item[(A4)] if $\{ U_i \}_{i\in I} \in \mathrm{Cov}_M$ then $\bigcup_{i\in I}
U_i \in \To (M)$ 
(the union of an admissible family is open),

\item[(A5)] if $\{ U_i \}_{i\in I} \in \mathrm{Cov}_M$, $V\subseteq
 \bigcup_{i\in I} U_i $, and $V\in \To (M)$, then $\{ V\cap  U_i \}_{i\in I}
 \in
 \mathrm{Cov}_M$
 (the traces of an admissible family on an open subset of the union
 of the family form an admissible family),

\item[(A6)] if $\{ U_i \}_{i\in I} \in \mathrm{Cov}_M $ and for each $i\in I$
 there is $\{ V_{ij} \}_{j\in J_i} \in \mathrm{Cov}_M$ such that
 $\bigcup_{j\in J_i}
V_{ij} = U_i $, then $\{ V_{ij} \}_{\stackrel{i\in I}{ j\in J_i}}
 \in \mathrm{Cov}_M$ 
(members of all admissible coverings of members of an admissible
 family form together an admissible family),

\item[(A7)] if $\{ U_i \}_{i\in I} \subseteq \To (M)$, $\{ V_j \}_{j\in J} \in
 \mathrm{Cov}_M$, $\bigcup_{j\in J} V_j =\bigcup_{i\in I} U_i$,
 and  $\forall j\in J \: \exists i\in I : V_j
 \subseteq U_i$, then $\{ U_i \}_{i\in I} \in \mathrm{Cov}_M$
(a coarsening, with the same union, of an admissible family  is admissible),

\item[(A8)] if $\{ U_i \}_{i\in I} \in \mathrm{Cov}_M $, $V\subseteq
 \bigcup_{i\in I} U_i$ and $V\cap U_i \in \To (M)$ for each $i$, then
 $V\in \To (M) $
(if a subset of the union of an admissible family has open
 traces with members of the family, then the subset is open).
\end{description}

Generalized topological spaces may be identified with certain Grothendieck sites, where the underlying category is a full, closed on finite (in particular: empty) products and coproducts subcategory of the category of subsets $\mc{P}(M)$ of a given set $M$ with inclusions as morphisms, and the Grothendieck topology is subcanonical, contains all finite jointly surjective  families and satisfies some regularity condition. 
(See \cite{sheaves} for the definition of a Grothendieck site.
Considering such an identification we should remember the ambient category $\mc{P}(M)$.)
More precisely: the axioms (A1), (A2) and (A3) contain a stronger version of the identity axiom of the Grothendieck topology. It is natural, since in model theory and in geometry we love finite unions, finite intersections and finite coverings. The axiom (A4) may be called \textit{co-subcanonicality}. Together with subcanonicality, it ensures that admissible coverings are coverings in the traditional sense. (Subcanonicality is imposed by the notation of
\cite{LSS}. The axiom (A4), weaker than (A8), justifies the notation
 $\mathrm{Cov}_M (U)$ of \cite{LSS}). The next are: (A5) the stability axiom of the Grothendieck topology, followed by the transitivity axiom (A6). Finally, (A7) is the saturation property of the Grothendieck topology (usually the Grothendieck topology of a site is required to be saturated), and the last axiom (A8) may be called the \textit{regularity axiom}. Both saturation and regularity have a smoothing character. Saturation may be achieved by modifying any generalized topological space, and  regularity by modifying a locally definable space (see I.1, page 3 and 9 of \cite{LSS}).     
The reader should be warned that (in general) the closure operator does not exist  for the generalized topology.

A \textbf{strictly continuous mapping} between generalized topological
spaces is such a mapping that the preimage of an (open)  admissible covering is admissible, which implies that the preimage of an  open set is open. (So strictly continuous mappings may be seen as morphisms of sites.)
Inductive limits exist in the category $\gts$ of generalized topological spaces and their strictly continuous mappings (see I.2 in \cite{LSS}).

Generalized topological spaces help to introduce further notions of interest that are generalizations of corresponding semialgebraic notions
(we follow here \cite{WSS}).

A \textbf{function sheaf of rings over $R$} on a generalized topological space
$M$ is a sheaf $F$ of rings on $M$ 
(here the sheaf property is assumed only for admissible coverings)
such that for each $U$ open in $M$
the ring $F(U)$ is a subring of the ring of all functions
 from $U$ into $R$, and the restrictions of the sheaf
 are the set-theoretical restrictions of mappings.
A \textbf{function ringed space over $R$} is a pair $(M, O_M )$, where $M$ is a
 generalized topological space and $O_M $ is  a function sheaf of rings over $R$.
 We will say about \textbf{spaces} (over $R$) for short.
An \textbf{open subspace} of a space over $R$ is an open subset
 of its generalized topological space together with the function sheaf of the space
  restricted to this open set.
A \textbf{morphism} $f :(M, O_M )\rightarrow (N, O_N )$ of function ringed spaces over $R$ is a 
strictly continuous mapping $f: M\rightarrow N$ such that for each open subset $V$ of $N$ the  set-theoretical substitution $h \mapsto h\circ f$ gives a morphism of rings $f^{\#}_V: O_N (V) \rightarrow O_M (f^{-1} (V))$. 
(We could  express this by saying that $f^{\#} : O_N \rightarrow  f_{*}O_M$ is the morphism of sheaves of rings on $N$ over $R$ induced by $f$. However, if  we define for function sheaves
$$(f_{*}O_M) (V)=\{ h:V\to R |\: h\circ f\in O_M(f^{-1}(V))\},$$
then each $f^{\#}_V:O_N(V)\to f_{*}O_M (V)$ becomes just an inclusion.)
Inductive limits exist in the category $\spa(R)$ of spaces over any $R$ and their morphisms (cf. I.2 of \cite{LSS} and \cite{ap2}).
Notice that our category of spaces over $R$, being a generalization 
(by passing from the semialgebraic to the general o-minimal case) 
of the category of spaces from \cite{WSS},  does not use the general sheaf theory for generalized topological spaces (as \cite{LSS} does), but only a bit of a simplier ``function sheaf theory''.

The following basic example is a special case of a definable space defined in \cite{Dries}.

\begin{fe}
Each definable subset $D$ of $R^n$ has a natural structure of a function ringed space over $R$.
Its open sets in the sense of the generalized topology are (relatively) open
 definable subsets, admissible coverings are such open coverings that already finitely many open sets cover the union,
and on each open definable subset $O\subseteq D$ we take the ring $\mathcal{DC}_D (O)$
of all continuous definable  $R$-valued functions on $O$.
\emph{Definable sets will be identified with such function ringed spaces.}
 Notice that the topological closure of a definable set is definable,
so the topological closure operator restricted to the class of definable subsets of a definable set  $D$ can  be treated as the closure operator in the generalized topological sense. 
\end{fe}

We start to re-introduce the theory of locally definable spaces by generalizing the definitions from \cite{LSS}.

 An \textbf{affine definable space} over $R$ is a space over $R$
isomorphic to a definable subset of some $R^n$. (Notice that
morphisms of affine definable spaces are given by continuous definable maps between
definable subsets of affine spaces.)
 
The following example, not explicitely studied before, shows that it is important to consider  affine definable spaces as definable sets  ``embedded'' into their ambient affine spaces.

\begin{exam}[``Bad boy'']
Consider the semialgebraic (that is definable in the ordered field structure)  space $S^1_{angle}$  over $\mb{R}$  on the underlying set $S^1$ of $\mb{R}^2$ obtained by taking the generalized topology 
from the usual affine definable circle $S^1\subseteq \mb{R}^2$
and declaring the structure sheaf to contain the continuous semialgebraic functions of the  angle $\theta$ having period $2\pi$. The two semialgebraic spaces are different.
The usual circle $S^1$ is an ``affine model'' of $S^1_{angle}$: there exists an isomorphism of semialgebraic spaces over $\mb{R}$ (whose formula is not semialgebraic, since it involves a trigonometric function) transforming  the ``non-embedded circle'' $S^1_{angle}$ into the ``embedded circle'' $S^1$.
\end{exam}

A \textbf{definable space} over $R$ is a space over $R$ that has
 a finite open covering  by affine definable spaces.
Definable spaces were introduced by van den Dries in \cite{Dries}.
They admit  clear notions of a definable subset and of an open subset.
The definable subsets of a definable space form a Boolean algebra
generated by the open definable subsets; ``definable'' here  means ``constructible from the generalized topology''.
A \textbf{locally definable space} over $R$ is a space over $R$ 
that has an admissible covering  by affine definable open subspaces.
(So definable spaces are examples of locally definable spaces.)
 Each locally definable space is an inductive limit of a directed system of definable spaces in the category of spaces over a given $R$ (cf. \cite[I.2.3]{LSS}). The \textbf{dimension} of a locally definable space is defined as usual (cf. p. 37 of \cite{LSS}), and may be infinite. 
\textbf{Morphisms} of affine definable spaces, definable spaces and locally
 definable spaces over $R$ are their morphisms as spaces over $R$.
 So affine definable spaces, definable spaces and locally definable spaces form full
 subcategories $\ads(R)$, $\ds(R)$, and $\lds(R)$ of the category $\spa(R)$ of spaces (over $R$).

 A \textbf{locally definable subset} of a locally definable space is
 a subset having definable intersections with all open definable
subspaces. Such subsets are also considered as \textbf{subspaces}, the locally definable space of such a set is formed as an inductive limit of definable subspaces of the open definable spaces forming the ambient space (cf. I.3, p. 28 in \cite{LSS}.) 
A locally definable subset of a locally definable space is called 
\textbf{definable} if as a subspace it is a definable space.
(The definable subsets of a definable space are exactly the definable subsets of these spaces as locally definable ones.)

On locally definable spaces we often consider a topology in
 the traditional sense, called the \textbf{strong topology} (cf. p. 31 of \cite{LSS}),
taking the open sets from the generalized topology as the basis of the topology.
Nevertheless, we will usually work in the generalized topology.
This allows, in many cases, to omit the word ``definably'' applied to topological notions (as in ``definably connected'').
On a definable space the generalized topology generates both the strong topology and the definable ( i. e. ``constructible'')  subsets. Similarly, 
the locally definable subsets of a locally definable space are exactly the sets ``locally constructible'' from the generalized topology, where ``locally'' means ``when restricted to an open definable subspace''.
The closure operator of the strong topology restricted to the class of locally definable subsets may be treated as the closure operator of the generalized topology. 

The following new example gives some understanding of the variety of locally definable spaces even in the semialgebraic case. They are obtained by ``partial localization'', which generalizes passing to the ``localization'' $M_{loc}$ of a locally complete locally semialgebraic space $M$ (see I.2.6 in \cite{LSS}). 

\begin{exam}\label{halfloc}
Consider any o-minimal expansion $\mb{R}_{\mc{S}}$ of the field $\mb{R}$.
Take the admissible union (see \cite{ap2}) of real line open intervals $(-\infty ,n)$ over all natural $n$, which implies that this family is assumed to be admissible.
Then this space is definable ``on the left-hand side'', but only locally definable ``on the right-hand side''. The definable subsets are the finite unions of intervals (of any kind) that are bounded from above. The locally definable subsets are locally finite unions of intervals
that have only finitely many connected components  on the negative half-line. The structure sheaf consists of functions that are continuous definable on each of the intervals $(-\infty,n)$.
This space will be called $(\mb{R}_{\mc{S}})_{loc,+}$. Analogously we define the space $(\mb{R}_{\mc{S}})_{loc,-}$ as the admissible union of
the family $(-n,+\infty)$, for $n\in \mb{N}$.
(By taking the admissible union of the family $(-n,n)$ for $n\in \mb{N}$, we would get the usual ``localization'' $(\mb{R}_{\mc{S}})_{loc}$ of the real line 
$\mb{R}_{\mc{S}}$.) 
\end{exam}

As in \cite{LSS}, we have

\begin{exam}[cf. I.2.4 in \cite{LSS}]
 Any ``direct (generalized) topological sum'' of definable spaces (in the category of spaces over a given $R$) is a locally definable space.
\end{exam}

We call a subset $K$ of a generalized topological space $M$ \textbf{small}
 if for each admissible covering $\mc{U}$ of any open $U$, the set $K\cap U$ is covered by finitely many members of $\mc{U}$.
(We say that $\mc{U}$ is \textbf{essentially finite} on $K$ in such a situation.)
Just from the definitions, we get (as in the semialgebraic case):

\begin{fs} Each definable space is small. Each subset of
 a definable space is also small. Every small open subspace of a locally definable space is definable.
Each small set of a locally definable space is contained in a small open set.
In particular ``small open'' means exactly ``definable open'', but ``small'' does not imply ``definable''.
\end{fs}

On can easily check that:
any locally definable space is topologically Hausdorff iff it is \textbf{Hausdorff} in the generalized topological sense.
Similarly, a locally definable space is topologically regular iff it is \textbf{regular} in the generalized topological sense: any single point, always closed, and any closed subspace not containing the point can be separated by disjoint open subspaces.

Clearly, each affine definable space is regular.
Of great importance for the theory of definable spaces is the following

\begin{thm}[Robson\cite{R}, van den Dries\cite{Dries}]
 Each regular definable space is affine.
\end{thm}

\begin{rem} \label{6}
 Even if we define locally definable spaces with the use of structure sheaves, a locally definable space is determined by its generalized topology when we assume silently that the structure of each affine subspace is understood, since it has an admissible covering of regular small open subspaces, which are affine definable spaces. The main purpose of introducing function ringed spaces was to define morphisms. 
\end{rem}


A practical way of defining and denoting a locally definable space is to write it as the \textit{admissible union} ($\stackrel{a}{\bigcup}$) of its admissible covering by open  definable (often affine) subspaces, not just the union set (even if considered with a topology). Such a notation is explored in \cite{ap2}. One can also just specify an admissible covering of the space by known open subspaces.

The author considers an attempt to encode the generalized topology under the notion of ``equivalent atlases'' a little bit risky.  We have the following important example, which is again obtained by the ``localization''
process known from \cite{LSS}.

\begin{exam}
Take an archmedean $R$. Consider three locally definable spaces $X_1,X_2,X_3$ on the same open interval $(0,1)$ given, respectively,  by admissible families of open definable sets $\mc{U}_1=\{  (\frac{1}{n}, 1-\frac{1}{n}): n\geq 3\} $, $\mc{U}_2=\{ (0,1)\} $, $\mc{U}_3=\mc{U}_1 \cup \mc{U}_2$. 
Then $X_1\neq X_2=X_3$.
Such a space $X_1$ is the ``localized'' unit interval $(0,1)_{loc}$.

\end{exam}

We would have two non-equivalent atlases $\mc{U}_1,\mc{U}_2$ that combine to a third atlas $\mc{U}_3$, and
the combined atlas $\mc{U}_3$ would be  equivalent to $\mc{U}_2$,  but not equivalent to $\mc{U}_1$.


\vspace{2mm}
Notice that the recent paper \cite{ldh} by E. Baro and M. Otero can easily mislead the reader. They define a locally definable space as a set with a concrete atlas, call some atlases equivalent (which is not studied later), and in Theorems 3.9 and 3.10 say that a set with only a topology is a locally definable space. The reader gets the impression that they consider only the usual topology, and do not see the essential use of the generalized topology (see the proof of (iii) of their 
Proposition 2.9). Their notion of an ``ld-homeomorphism'' is never defined, and the reader may wrongly guess that an ld-homeomorphism is just a locally definable homeomorphism  (see Remark 2.11).
 Their ``locally finite generalized simplicial complex'' is given a locally definable space structure ``star by star'', so it is not necessary ``embedded'' into the ambient affine space. This may mislead the reader when reading their version of the Triangulation Theorem (Fact 2.10) and some proofs. Their Example 3.1 is highly imprecise, since it depends on the choice of the covering of $M$ by definable subsets $M_i$. The same symbol $M$ denotes both a locally definable space and just a subset of $R^n$ (and this is continued in their Example 3.3).
Formula $Fin(\mb{R})=\mb{R}$ (see p. 492) again suggests to the reader the nonexistence of the generalized topology (never mentioned explicitely).
Its worth noting that if $R$ does not have any saturation (as in the important case of the field of real numbers $\mb{R}$), then the usual topology does not determine the generalized topology.

We will say that an object $N$ of $\lds (R)$ \textbf{comes from} $R^k$ if the underlying topological space of $N$ is equal to the standard topological space of $R^k$  and  for each $x\in R^k$ both $N$ and the affine space $R^k$ induce on an open box $B$ containing $x$ the same definable open subspace.
The following two original theorems show the variety of locally definable spaces ``living'' on the same topological space.

\begin{thm}\label{realline}
For each o-minimal expansion $\mb{R}_{\mc{S}}$ of the field of real numbers $\mb{R}$:

1) there are exactly four different objects of $\lds(\mb{R}_{\mc{S}})$ that come from $\mb{R}_{\mc{S}}^1$, namely: $\mb{R}_{\mc{S}}$, $(\mb{R}_{\mc{S}})_{loc}$,
$(\mb{R}_{\mc{S}})_{loc,+}$, $(\mb{R}_{\mc{S}})_{loc, -}$.

2) there are uncountably many different objects of $\lds(\mb{R}_{\mc{S}})$ that come from $\mb{R}^2_{\mc{S}} $.
\end{thm}

\begin{proof}
1) Assume $N$ is an object of $\lds(\mb{R}_{\mc{S}})$ coming from $\mb{R}_{\mc{S}}^1$. Each open subset of $N$ is a countable union of open intervals. Each open definable set is a finite union of open intervals, since it has a finite number of connected components. There is an admissible covering $\mc{U}$ of the real line by open intervals that are affine definable spaces. If such an interval is bounded, then it is relatively compact, and a standard
interval (that is an interval as an ``embedded'' subspace of $\mb{R}_{\mc{S}}$).
If such an interval is not bounded, then each bounded subinterval is standard, hence again the whole interval is standard, since it is an affine definable space.
If there are no infinite intervals in $\mc{U}$, then the open family $\{ (-n,n)\}_{n\in \mb{N}}$ is admissible, and $N=(\mb{R}_{\mc{S}})_{loc}$.
If both  $+\infty$ and $-\infty$ are ends of intervals from  $\mc{U}$ , then $N$ is a finite union of standard intervals, thus it is isomorphic to the affine space $\mb{R}_{\mc{S}}$. Similarly, the two other cases give the spaces  $(\mb{R}_{\mc{S}})_{loc,+}$, $(\mb{R}_{\mc{S}})_{loc, -}$.

 


2) Choose a slope $a\in \mb{R}$ and consider the space
$N_a$ defined by the admissible covering
$ \{ U_{a,n} \}_{n\in \mb{N}}$, where $U_{a,n}=\{ (x,y)\in \mb{R}^2:\: y<ax+n\}$ are definable sets.
All $N_a$, for $a\in \mb{R}$, are different objects of $\lds(\mb{R}_{\mc{S}})$. 
\end{proof}

Remind (from non-standard analysis) that each non-archimedean $R$ is partitioned into many \textbf{galaxies} (two elements $x,y\in R$ are in the same galaxy if their ``distance'' $|x-y|$ is bounded from above by a natural number).

\begin{thm}\label{nonreal}
For any o-minimal expansion $R$ of a field  not isomorphic to $\mb{R}$
 there are already uncountably many different objects of $\lds(R)$ that come from the line
 $R^1$.
\end{thm}
\begin{proof}
\textit{Case 1: $R$ contains $\mb{R}$.}

The set of galaxies of $R$ is uncountable. For any galaxy $G$ of $R$ take $x\in G$ and consider the space $N_G$ defined as the disjoint generalized topological union of the following: all the galaxies $G' >G$, treated each one as a  locally definable space (see Remark \ref{galaxy}), and the space
$N^{'}_G$ given by the admissible covering  $ \{(-\infty, x+n)\}_{n\in \mb{N}}$, which is the union of all galaxies $G''\leq G$ ``partially localized'' (only) at the end of $G$.

All of $N_G$ are different objects of $\lds(R)$ and come from $R^1$.

\textit{Case 2: $R$ does not contain $\mb{R}$.}

The field $R\cap \mb{R}$ has uncountably many irrational cuts, determined by elements $r\in \mb{R}\setminus R$. For each such $r$, consider the space
$N_r$ over $R$ defined by the admissible covering
$$ \{ (-\infty, s) \}_{s<r} \cup \{  (s,+\infty)\}_{s>r}, $$
where $s\in R \cap \mb{R}$.
This space consists of two connected components given by the conditions $x<r$ and $x>r$.
All of $N_r$, $r\in \mb{R}\setminus R$, are different objects of $\lds(R)$
and come from $R^1$.
\end{proof}

There are more general sets that are called in \cite{Fischer}, Definition 7.1(a), ``locally definable''. We will call them local subsets.  (A subset $Y$ of a space $X$ is a \textbf{local  subset} if for each point $y$ only of $Y$ there is an open definable neighborhood $U$ of $y$ in $X$ such that $U\cap Y$ is definable.)
They can be given  a locally definable space structure, but their properties are not  nice: they are closed only on finite intersection, and are not closed under complement or even finite union.
   
The locally definable space on such $Y\subseteq X$ may be introduced by the following admissible covering
$$ \mc{U}_Y = \{ Y\cap U_i \mid U_i\mbox{ is a definable open subset of $X$ and }Y\cap U_i  \mbox{ is definable}\} .$$
(The above definition does not depend on any arbitrary  choice of an admissible covering, contrary to Example 3.1 and Example 3.3 of \cite{ldh}.)

Local subsets are \underline{not} (as such) called subspaces!
Their use often does not recognize  the  space structure given above
  (even if they are definable sets), since we mainly want to study ``locally definable functions'' on them (see Definition 7.1 (b) of \cite{Fischer}).
(Consider a function ``locally definable'' if its domain and codomain are local subsets of some objects of $\lds(R)$, and all function germs of this function at points of its domain are definable. A function germ $f_x$  at $x$ is called definable if some definable neighborhood of $x$ is mapped by $f$ into a definable neighborhood of $f(x)$ and the obtained restriction of $f$ is a definable mapping.)

The following  examples make the above considerations more clear. 

\begin{exam}
The semialgebraic set $(-1,1)_{\mb{R}}$ inherits an affine semialgebraic space structure from $\mb{R}$.  Nevertheless, when speaking about ``locally semialgebraic functions'' into $\mb{R}$ (in the sense of  Definition 7.1 (b) of \cite{Fischer}) we want to treat it as the ''localized'' open interval $(-1,1)_{loc}$, which is not a  semialgebraic space.
Define, for example, functions $w:\mb{R}\to \mb{R}$ and $u:(-1,1)\to \mb{R}$
by formulas
$$w(x)=\left\{ \begin{array}{ll}
x-4k, & x\in [4k-1,4k+1), k\in \mb{Z},\\
2+4k-x, & x\in [4k+1,4k+3), k\in \mb{Z},
\end{array} \right.$$
and
$$ u(t)=w(\frac{t}{\sqrt{1-t^2}}).$$
Then $u$ is "locally semialgebraic" (and not semialgebraic).  
\end{exam}

\begin{exam}
Consider the semialgebraic set $S=(-1,1)^2\cup \{ (1,1)\}$ in $\mb{R}^2$. The fact of being a ``locally semialgebraic function'' (in the sense of Definition 7.1 (b) of \cite{Fischer}) on $S$ (into $\mb{R}$) does not reduce to being a morphism
of any locally (and even weakly) semialgebraic space that can be formed by redefining the notion of an admissible covering of the space $S$.
In particular, each of the functions $F_n:S\to \mb{R}$ $(n=1,2,3,...)$,
where 
$$F_n (x,y)=\left\{ \begin{array}{ll}
0, & y\geq 1-\frac{1}{n},\\
w(\frac{1-\frac{1}{n}-y}{1-x}), &  y<1-\frac{1}{n},
\end{array}\right.$$
is "locally semialgebraic" (function $w$ is defined as in the previous example).
\end{exam}

In general, definable spaces and locally definable spaces do not behave well enough for being used in homotopy theory. The right choice of assumptions (as in the semialgebraic case of \cite{LSS}) are: regularity
and one new called ``paracompactness'', which is only an analogue of the topological notion.

\section{Regular paracompact locally definable spaces}
 
One of the  reasons why we pass to the locally definable spaces 
 is the need of existence of covering mappings with infinite (for example countable) fibers.

The following example is a generalization of an example from \cite{DK6}.

\begin{exam}[cf. 5.14 in \cite{DK6}]\emph{The space $\textrm{Fin} (R)$.} 
We look for the \emph{universal covering} of the unit circle $S^1 \subseteq R^2$.
We will see soon that (as in topology) $\pi_1 (S^1 )=\mathbb{Z} $, so the universal covering should have countable fibers.
Let $\textrm{Fin} (R)$ be the locally definable space
introduced by the admissible covering by open intervals 
$\{(-n,n)\}_{n\in \mathbb{N}}$ in $R$. 
 There is a surjective semialgebraic (so definable) morphism  $e:[0,1]\rightarrow S^1$
 that maps 0 and 1 to
the distinguished point on $S^1$ and is injective elsewhere. Then the universal covering mapping
$p: \textrm{Fin}(R)
\rightarrow S^1$ defined by $p(m+x)=e(x)$,
where $m\in \mathbb{Z}, x\in [0,1]$,  is a morphism
of locally definable spaces.
\end{exam}

A family of subsets of a locally definable space is \textbf{locally finite} if
 each open definable subset of the
space meets only finitely many members of the family.

A locally definable space is called \textbf{paracompact} if there is a
 locally finite covering of the whole space by open definable subsets.
 (A locally finite covering must be admissible,  since ``admissible'' means: when restricted to an open definable subspace, there is a finite subcovering. Shortly: ``admissible'' means exactly ``locally \textsl{essentially} finite''.)

\begin{rem}\label{galaxy}
The locally definable space $\textrm{Fin} (R)$ given by the
admissible covering  $\{ (-n,n): n\in \mathbb{N}\} $ is paracompact for each $R$, since there exists a locally finite covering giving the same space. 
(Notice that if $R$ contains $\mb{R}$, then $\bigcup\limits^a_{r\in \mb{R}_{+}} (-r,r) =\bigcup\limits^a_{n\in \mb{N}} (-n,n)$.)
 In the language of nonstandard analysis, we can say that each galaxy may be considered as a regular paracompact locally definable space.
\end{rem}

Direct (i. e. cartesian) products preserve regularity and paracompactness of locally definable spaces (cf. I.4.2c) and I.4.4e) in \cite{LSS}). We will denote the category of regular paracompact locally definable spaces over $R$ by  $\rplds(R)$.
\begin{exam}
The spaces from the proofs of Theorems \ref{realline} and \ref{nonreal} are objects of $\rplds(R)$.
\end{exam} 

\begin{rem}
A \textbf{connected} (in the sense of  generalized topology: the space cannot be decomposed into two open disjoint nonempty subspaces)  regular paracompact locally definable space has a countable admissible covering by definable open subsets (so called \textbf{Lindel\"of property} in \cite{LSS}). 
If it has finite dimension $k$, then it can be embedded into the cartesian power  Fin$(R)^{2k+1}$. This holds by embedding into a partially complete space, triangulation (see Theorems \ref{embed}, \ref{triang}   below) and  Theorem    3.2.9  from a book of Spanier \cite{spa} (see also II.3.3 of \cite{LSS}).
\end{rem} 

\begin{toprem}
 The notion of paracompactness introduced above differs from
the topological one. Each definable space is paracompact. There are Hausdorff definable (so paracompact) spaces which are not regular. With the regularity assumption, each paracompact space is normal and admits partition of unity. Paracompactness is inherited by all subspaces and cartesian products. The Lindel\"of property gives paracompactness only with the assumption that the closure of a definable set is definable.
\end{toprem}

Fiber products exist in the category of locally definable spaces over $R$ (cf. I.3.5 of \cite{LSS}).
A morphism $f:M\rightarrow N$ between locally definable spaces
is called \textbf{proper} if it is universally closed in the sense of the
 generalized topology. This means
that for each morphism of locally definable spaces $g:N' \rightarrow N$
the induced morphism $f' : M\times_{N} N' \rightarrow N'$ in the pullback
 diagram is a closed
mapping in the sense of the generalized topological spaces (it maps
 closed subspaces onto closed subspaces). If all restrictions of $f$ to closed definable subspaces are proper, then we call $f$ \textbf{partially proper}. 

A Hausdorff locally definable space $M$ is called \textbf{complete} if the 
 morphism from $M$ to the one point space is proper. 
 Each paracompact complete space is affine definable (compare I.5.10 in \cite{LSS}).
Moreover, $M$ is called \textbf{locally complete} if each point has a complete neighborhood. (Each locally complete locally definable space is regular, cf. I.7, p. 75 in \cite{LSS}). It is  \textbf{partially complete} if every closed definable subspace is complete. Every partially complete regular space is locally complete (cf. I.7.1 a)) in \cite{LSS}).

\begin{toprem}
 This notion of properness is analogical to a notion 
from algebraic geometry. Partial completeness is the key notion. 
\end{toprem}

 Let $M$ be a locally complete paracompact space.
 Take the family $\stackrel{\circ}{\gamma}_c (M)$ of all such open definable subsets $U$ of $M$ that $\overline{U}$ is complete. Introduce a new locally definable space
$M_{loc}$, the \textbf{localization} or \textbf{partial completization} of $M$, on the same underlying set  taking  $\stackrel{\circ}{\gamma}_c (M)$ as an admissible covering by small open subspaces (cf. I.2.6 in \cite{LSS}).  The new space is regular partially complete (not only locally complete) and the identity mapping from $M_{loc}$ to $M$ is a morphism, but $M_{loc}$ may not be paracompact,
see Warning-Example \ref{rloc}. 
Notice that localization leaves the strong topology unchanged. 

\begin{toprem}
 Localization is similar  to the process of
passing to $k$-spaces (they are exactly the compactly generated spaces if hausdorffness is assumed) in homotopy theory. (Complete spaces play the role of compact spaces.) But notice that each topological locally compact space is a $k$-space.
\end{toprem}

\begin{rem}
Only one of the four locally definable spaces of Theorem \ref{realline} for each $\mb{R}_{\mc{S}}$ is partially complete, namely $(\mb{R}_{\mc{S}})_{loc}$.
\end{rem}

 A \textbf{paracompact locally definable manifold} of  dimension $n$ over $R$ is
a Hausdorff locally definable space over $R$ that has  a locally finite covering
by definable open subsets that are isomorphic to open balls in $R^n$.
(Such a space is paracompact and locally complete, so regular, cf. I.7 p.75 in \cite{LSS}.)
If additionally the transition maps are (definable) $C^k$-diffeomorphisms ($k=1,...,\infty$), then  we  get \textbf{paracompact locally definable $C^k$-manifolds}. Notice that the differential structure of such manifolds may be encoded by sheaves  (in the sense of the strong topology) of $C^k$ functions. We get the following original result:

\begin{thm}\label{manifold}
 Paracompact (in the topological sense) analytic manifolds of dimension $n$ are in bijective correspondence with partially complete paracompact locally definable $C^{\infty}$-manifolds over $\mathbb{R}_{an}$ of the same dimension.
\end{thm}

\begin{proof}
\textit{A paracompact analytic manifold induces a  paracompact locally definable
 $C^{\infty}$-manifold over $\mathbb{R}_{an}$}: Each paracompact manifold (even a topological one) is regular.
We may assume (by shrinking the covering of the manifold  by chart domains if necessary) that the analytic structure of the manifold is given by a locally finite atlas consisting of charts whose domains and ranges are relatively compact subanalytic sets, and the charts extend analytically beyond the closures of chart domains. By taking a nice locally finite refinement, we additionally can get the chart domains and chart ranges (analytically and globally subanalytically) isomorphic to open balls in $\mb{R}^n$. 
Now the chart domains form a locally finite covering of the analytic manifold  that defines a paracompact locally definable manifold over $\mb{R}_{an}$. The transition maps (being analytic diffeomorphisms) are $\R_{an}$-definable $C^{\infty}$-diffeomorphisms 
of open, relatively compact, subanalytic subsets of some $\mathbb{R}^n$.
Thus we get a locally definable $C^{\infty}$-manifold 

Notice that the relatively compact subanalytic sets are now the definable sets and
the subanalytic sets are now the locally definable sets.

The obtained locally definable space is partially complete.
 
\textit{Vice versa}:
 A  paracompact locally definable $C^{\infty}$-manifold over
$\mathbb{R}_{an}$ induces a Hausdorff (analytic) manifold with analytic, globally subanalytic transition maps and globally subanalytic chart ranges.
We may assume that the manifold is connected.
Its locally finite atlas is countable (cf. I.4.17 in \cite{LSS}),
 so the manifold is a second countable topological space, and finally  a paracompact analytic manifold.
All locally definable subsets are now subanalytic (they are globally subanalytic in every chart).
 
\textit{One-to-one correspondence}:
If the paracompact locally definable manifold is partially complete, then
the closure of a chart domain is a closed definable set (cf. I.4.6 of \cite{LSS}) and a complete definable set, which means it is a compact subanalytic set. Thus chart domains are relatively compact. ``Locally'' in the sense of locally definable spaces means exactly ``locally'' in the topological sense.
It follows that: the definable subsets are exactly the relatively compact subanalytic subsets,
 and the locally definable subsets are exactly the subanalytic subsets of the obtained paracompact analytic manifold. Notice that the strong topology does not change when we pass from one type of a manifold to the other.
 So the structure of the partially complete locally definable space is uniquely determined (see Remark \ref{6}).
Both the structures of a $C^{\infty}$ locally definable manifold over $\mb{R}_{an}$ and the structure of an analytic manifold do not change during the above operations (only a convenient atlas was chosen).
\end{proof}

\begin{rem} \label{analytic}
 A real function on a (paracompact) analytic manifold $M_{an}$ is analytic iff it is a $C^{\infty}$ morphism from the corresponding partially complete paracompact locally definable $C^{\infty}$-manifold (call it $M_{ldm}$) into $\mathbb{R}_{an}$ as an affine definable space. (See 5.3 in \cite{DM}.)
\end{rem}

Analogously, for each expansion $\mb{R}_{\mc{S}}$ of the field $\mb{R}$ that is a reduct of $\mb{R}_{an}$, partially complete paracompact locally definable $C^{\infty}$-manifolds over  $\mb{R}_{\mc{S}}$ correspond uniquely to  paracompact analytic manifolds of some special kinds. Then the locally definable subsets in the sense of a given locally definable manifold (as well as in the sense of its ``expansions'', see below) form nice ``geometric categories''. This in particular generalizes the \textit{analytic-geometric categories} of van den Dries and Miller \cite{DM}.

The above phenomenon may be explained in the following way: the analytic manifolds $\mb{R}^n$
($n\geq 1$), which model all analytic manifolds, have a natural notion of smallness. A subset $S\subset \mb{R}^n$ is \textbf{topologically small} if it is \textit{bounded} or, equivalently,  \textit{relatively compact}. In the corresponding partially complete paracompact locally definable $C^{\infty}$-manifolds $Fin(\mb{R}_{an})^n=(\mb{R}_{an})_{loc}^n = \bigcup\limits^a_{k\in \mb{N}} \: (-k,k)^n = Fin((\mb{R}_{an})^n)$ over $\mb{R}_{an}$  this means that $S$ is a \textit{small} subset in the sense of the generalized topology
(if $S$ is subanalytic, then this means \textit{definable}). One could also use  the notion of being \textit{relatively complete} in this context. 
It is partial completeness that gives analogy between the usual topology and the generalized topology.

\begin{rem}
The generalized topology of the space $M_{ldm}$ of Remark \ref{analytic} is ``the subanalytic site'' considered by microlocal analysts (see \cite{KS}).
More generally:  the generalized topology of each paracompact locally definable manifold may be considered as a ``locally definable site''.
It is also possible to consider all subanalytic subsets of a real analytic manifold as open sets of a generalized topological space. But then the strong topology becomes discrete.
\end{rem}

\begin{ew}[cf.  I.2.6 in \cite{LSS}] $($\emph{The space $R_{loc}$.}$)$ \label{rloc}
The structure $R$ as an affine definable space,  is locally complete but not complete.  For such a  space  $R_{loc}$ is
introduced by the admissible covering
 $\{ (-r,r): r\in R_{+} \} $.
 This is a locally (even regular partially) complete space which is not definable.
If the cofinality of $R$ is uncountable, then $R_{loc}$ is not paracompact!
Here the morphisms from $R$ to
$R$ are ``the continuous definable functions'', and the
morphisms from ${R}_{loc}$ to $R$ are ``the continuous
 locally (in the sense of $R_{loc}$) definable functions''. (The latter case  includes
 some nontrivial periodic funtions for an archimedean $R$.)
 \end{ew}

A series of topological facts have counterparts for regular paracompact locally definable spaces.

\begin{lem}[cf. \cite{DK6} and \cite{LSS}, chapter I] \label{15}
Let $M$ be an object of $\rplds(R)$. Then:

a) \textbf{[tautness]} the closure of a definable set is definable (cf. I.4.6);

b) \textbf{[shrinking of coverings lemma]} for each locally finite covering
 $(U_{\lambda} )$ of $M$ by open locally definable sets 
there is a covering $(V_{\lambda} )$ of $M$ by open locally definable
 sets such that $\overline{V_{\lambda} }\subseteq
U_{\lambda}$ (cf. I.4.11);

c) \textbf{[partition of unity]} for every locally finite covering
 $(U_{\lambda} )$ of $M$ by open locally definable
subsets there is a subordinate partition of unity, i. e. there is a
 family of morphisms
 $\phi_{\lambda} : M\rightarrow [0,1]$ such that $\mathrm{supp}
 \phi_{\lambda} \subseteq U_{\lambda}$
and $\sum_{\lambda} \phi_{\lambda} =1$ on $M$ (cf. I.4.12);

d) \textbf{[Tietze's extension theorem]} if $A$ is a closed subspace of $M$
 and $f:A\rightarrow K$ is a morphism into a
convex definable subset $K$ of $R$, then there exists a morphism
 $g:M\rightarrow K$ such that $g|A =f$ (cf. I.4.13);

e) \textbf{[Urysohn's lemma]} if $A,B$ are disjoint closed locally
 definable subsets of $M$, then there is a
morphism $f:M\rightarrow [0,1]$ with $f^{-1} (0)=A$ and $f^{-1} (1)=B$ (cf. I.4.15).
\end{lem}

Each  locally definable space $M$ over $R$ has a natural ``base field extension'' $M(S)$
 over any elementary extension $S$ of $R$ (cf. I.2.10 in \cite{LSS}) and an ``expansion'' $M_{R'}$ to a locally definable space over any o-minimal expansion $R'$ of $R$. 
Analogously, we may speak about a base field extension of a morphism.

The rules of conservation of the main properties under the base field extension are the same as for the locally semialgebraic case:  
\begin{itemize}
\item[a)] the base field extensions of the family of the connected components of a locally definable space $M$ form the family of connected components of $M(S)$ (cf. I.3.22 i) in \cite{LSS});
\item[b)] if $M$ is Hausdorff then: the space $M$ is definable iff $M(S)$ is definable, $M$ is affine definable iff $M(S)$ is affine definable, $M$ is paracompact iff $M(S)$ is paracompact,
$M$ is regular and paracompact iff $M(S)$ is regular and paracompact (cf. B.1 in \cite{LSS}).
\item[c)] if $M$ is regular and paracompact, then:
$M$ is partially complete iff $M(S)$ is partially complete,
$M$ is complete iff $M(S)$ is complete (cf. B.2 in \cite{LSS}).
\end{itemize}

If we expand $R$ to an o-minimal $R'$ then:
\begin{itemize}
\item[a)]  any locally definable space $M$ is  regular over $R$ iff $M_{R'}$ is a regular space over $R'$, since they have the same strong topologies;  
\item[b)] a locally definable space $M$ is connected over $R$ iff $M_{R'}$ is connected over $R'$ (for an affine space a clopen subset of a set definable over $R$ is definable over $R$, generally apply an admissible covering by affine subspaces ``over $R$'');
\item[c)] a locally definable space $M$ is Lindel\"of over $R$ iff $M_{R'}$ is Lindel\"of over $R'$: if $M$ is  Lindel\"of, then $M_{R'}$ is obviously Lindel\"of; if $M_{R'}$ is  Lindel\"of
then each member of a countable admissible covering $\mc{V}$ of $M_{R'}$ by definable open subspaces is covered by a finite union of elements of the admissible covering $\mc{U}$ of $M$ by definable open subspaces that allowed to construct $M_{R'}$. Then $\mc{U}$ has a countable 
subcovering $\mc{U'}$. 
(Up to this moment our proof goes like the proof of
Proposition 2.9 iii) in \cite{ldh}, but they  do not care about admissibility.)
The family $\mc{U''}$ of finite unions of elements of $\mc{U'}$ is a countable coarsening of $\mc{V}$, hence is admissible in $M_{R'}$. Since ``admissible'' means ``locally essentially finite'', $\mc{U''}$ is in particular admissible in $M$. 

\item[d)] if  $M$ is a Hausdorff locally definable space over $R$, then $M$ is paracompact over $R$ iff $M_{R'}$ is paracompact over $R'$: if $M$ is paracompact, then $M_{R'}$ is obviously paracompact; if $M_{R'}$ is paracompact, then we can assume that it is connected. Then $M_{R'}$
is Lindel\"of (cf. I.4.17 in \cite{LSS}) and {taut} (i.e. the closure of a definable set is definable, cf. I.4.6 in \cite{LSS}). Now, by c),
the space $M$ is Lindel\"of, and it is taut by the construction of $M_{R'}$
and considerations of Fundamental Example 1, so $M$ is paracompact (see  I.4.18  in \cite{LSS} and Proposition 2.9 iv) in \cite{ldh}). 
\end{itemize}

\section{Homotopies}
Here basic definitions of homotopy theory are re-introduced.
The unit interval $[0,1]$ of $R$ will be considered as an affine definable space over $R$.

Let $M,N$ be objects of $\spa(R)$  and let $f,g$ be morphisms
 from $M$ to $N$.
A \textbf{homotopy} from $f$ to  $g$ is a morphism $H:M\times
 [0,1]\rightarrow N$ such that
$H(\cdot ,0)=f$ and $H(\cdot ,1)=g$. If $H$ exists, then $f$ and $g$ are called \textbf{homotopic}. If additionally $H(x,t)$ is independent of $t\in [0,1]$ for each $x$ in a subspace $A$, then we say that $f$ and $g$ are \textbf{homotopic relative to} $A$.
A  subspace $A$ of a space  $M$ is called a \textbf{retract} of $M$ if
 there is a morphism
$r:M\rightarrow A$ such that $r|A=id_A$. Such $r$ is called a
 \textbf{retraction}.
A subspace $A$ of $M$ is called a \textbf{strong deformation retract} of $M$ if
there is a homotopy $H:M\times [0,1]\rightarrow M$ such that $H_0$
 is the identity
and $H_1$ is a retraction from $M$ to $A$. Then $H$ is called \textbf{strong deformation retraction}.

A \textbf{system of spaces} over $R$ is any tuple $(M,A_1 ,...,A_k )$
where $M$ is a space over $R$ and $A_1,...,A_k$ are subspaces
 of $M$. A \textbf{closed pair} is a system $(M,A)$ of a space with a closed subspace.
A system $(A_0,A_1,...,A_k)$ is \textbf{decreasing} if $A_{i+1}$ is a subspace of $A_i$ for $i=0,...,k-1$.
A \textbf{morphism} of systems of spaces $f:(M,A_1 ,...,A_k )
\rightarrow (N, B_1 ,...,B_k )$ is a morphism of spaces
 $f:M\rightarrow N$ such that $f(A_i )\subseteq B_i$ for each $i=1,...,k$.
A \textbf{homotopy} between two morphisms of systems of spaces
$f,g$ from $(M,A_1 ,...,A_k )$ to $(N,B_1 ,...,B_k )$ is a morphism
$$ H:(M\times [0,1],A_1 \times [0,1],...,A_k \times[0,1])\rightarrow
 (N,B_1 ,...,B_k )$$
with $H_0 =f$ and $H_1 =g$.
The \textbf{homotopy class} of such a morphism $f$ will be denoted by $[f]$ and
the \textbf{set of all homotopy classes} of morphisms from $(M,A_1 ,...,A_k )$
to $(N,B_1 ,...,B_k )$ by $$ [ (M,A_1 ,...,A_k ), (N,B_1 ,...,B_k ) ] .$$
If $C$ is a closed subspace of $M$, and $h:C\rightarrow N$ is a pregiven morphism such that $h(C \cap A_i ) \subseteq B_i$, then we denote the sets of classes of homotopy relative to $C$ of mappings extending $h$ by
$$[ (M,A_1 ,...,A_k ), (N,B_1 ,...,B_k ) ]^{h} .$$

Let us adopt the notation:
$I=[0,1]$, $\partial I^n =I^n \setminus (0,1)^n$, and
$J^{n-1} = \overline{\partial I^n \setminus (I^{n-1} \times \{ 0\}  )}$.
For every pointed space $(M, x_0 )$ over $R$ and $n\in \mathbb{N}^{*}$
 we define the  \textbf{(absolute) homotopy groups} as sets
$$ \pi_n (M,x_0 )=[ (I^n ,\partial I^n ), (M,x_0 ) ] $$
where the multiplication $[f]\cdot [g]$, for $n\geq 1$, is the homotopy class of
$$ (f * g) (t_1 ,t_2 ,...,t_n )=\left\{
\begin{array}{l}
f (2 t_1 ,t_2 ,...,t_n ), 0\le t_1 \le \frac{1}{2} \\
g (2 t_1 -1,t_2 ,...,t_n ), \frac{1}{2} \le t_1 \le 1.
 \end{array} \right. $$
For $n=0$ we get (only) a set $\pi_0 (M,x_0 )$ of connected components
of $M$ with the base point the connected component of $x_0$.
Also, as in topology,  we define \textbf{relative homotopy groups}
$$\pi_n (M,A,x_0 ) = [(I^n ,\partial I^n ,J^{n-1} ),(M,A,x_0 )].$$

A morphism $f:M\rightarrow N$ is a
\textbf{homotopy equivalence} if there is a morphism $g:N\rightarrow M$ such that $g\circ f$ is homotopic to $id_M$ and $f\circ g$ is homotopic to $id_N$. We call $f:M\rightarrow N$
 a \textbf{weak homotopy equivalence} if $f$  induces bijections in homotopy sets ($\pi_0 (\cdot)$) and group isomorphisms in all homotopy groups ($\pi_n (\cdot)$, $n\geq 1$).
Analogously, we define homotopy equivalences and weak homotopy equivalences for systems of spaces.

The following operations, known from the usual homotopy theory, may not be executable in the category of regular paracompact locally definable spaces over a given $R$.
The \textbf{smash product} of two pointed spaces $M,N$, which  is $M\wedge N = M\times N / M\vee N$, where $M\vee N$ denoted the wedge product of such spaces. 
The \textbf{reduced suspension} $SM$  of $M$, which  is $S^1 \wedge M$.
The \textbf{mapping cylinder} $Z(f)$ of $f:M\rightarrow N$, which  is the space obtained as the quotient of $(M\times [0,1]) \cup N$ by the equivalence relation 
 that identifies  each point of the form $(x,1)$, $x\in M$, with $f(x)$. The \textbf{mapping cone} of $f$, which  is the mapping cylinder of $f$ divided by $M\times \{ 0\} $.
The \textbf{cofiber} $C(f)$ of $f:M\rightarrow N$, which is the ``switched'' mapping cylinder $(([0,1]\times M)\cup_{1\times M,f} N) / \{ 0\} \times M$.

\section{Comparison Theorems for locally definable spaces}
In this section the two Comparison Theorems from \cite{LSS} are extended, and the third is added. The first steps to do this are: embedding in a partially complete space and triangulation.
 
\begin{thm}[embedding into a partially complete space, cf. II.2.1]\label{embed}
Each regular paracompact locally definable space over $R$ is isomorphic to a dense locally definable subset of a partially complete regular paracompact space over $R$.
\end{thm} 

We restate the triangulation theorem, keeping the notation from \cite{LSS} to avoid confusion.
 
 \begin{thm}[triangulation, cf. II.4.4]\label{triang}
 Let $M$ be a regular paracompact locally definable space over $R$. For a given locally finite family $\mc{A}$ of locally definable subsets of $M$, there is a simultaneous triangulation $\phi : X\rightarrow M$ of $M$ and $\mc{A}$ (i. e. an isomorphism from the underlying set $X$, considered as a locally definable space, of a strictly locally finite geometric simplicial complex $(X,\Sigma(X))$ to $M$ such that all members of $\mc{A}$ are unions of images of open simplices from $\Sigma(X)$). 
 \end{thm}
 
In particular, each object of $\rplds(R)$ is locally (pathwise) connected  and even locally contractible.

As an illustration of the methods available by triangulation, the following Bertini or Lefschetz type theorem (known from complex algebraic and analytic geometry) is proven. (See \cite{ap} for a topological version. Here the difficulty lies in the possibility that two different points are of an infinitesimal distance, and that a curve has 
an infinitely large velocity.)

A subspace $\Delta$ of a locally definable space $Y$  \textbf{nowhere disconnects} $Y$ 
if for each connected open neighborhood $W$ of any $y\in Y$ there is an open neighborhood $U\subseteq W$ of $y$ such that $U\setminus \Delta$ is connected.

A morphism $p:E\rightarrow B$ in \textbf{LDS($R$)} is a \textbf{branched covering} if there is a closed, nowhere dense \textbf{exceptional subspace} $\Delta\subseteq B$ such that 
$p|_{p^{-1}(B\setminus \Delta)}:p^{-1}(B\setminus \Delta)
\rightarrow B\setminus \Delta$ is a \textbf{covering mapping} (this means: there is an admissible covering of $B\setminus \Delta$  by open subspaces, each of them well covered, analogically to the topological setting).  
 If each of the \textbf{regular points} $b\in B\setminus \Delta$ of the branched covering $p:E\to B$ have the fiber of the same cardinality, then this cardinality is
 called the \textbf{degree} of a branched covering $p:E\to B$.

\begin{thm}[cf. Thm. 1 in \cite{ap}] \label{bertini}
Let $Y$ be a simply connected (this assumes connected) object of $\rplds(R)$, $Z$ be a connected, paracompact locally definable manifold over $R$ of dimension at least 2, and $\pi :Y\times Z\to Y$ the canonical projection.

Assume that $V\subset Y\times Z$ is a closed subspace such that the restriction $\pi_V :V\to Y$ is a branched covering of finite degree and an  exceptional set $\Delta$ of this branched covering nowhere disconnects $Y$. Put $X=(Y\times Z)\setminus V$, and $L=\{ p\} \times Z$, for some $p\in Y\setminus \Delta$.

If there is a morphism of locally definable spaces $h:Y\to Z$ over $R$ with the graph contained in $X$, then the inclusion $i:L\setminus V \to X$ induces an epimorphism in the fundamental
groups $i_{*}: \pi_1 (L\setminus V)\to \pi_1 (X)$. 
\end{thm}

\begin{lem}[straightening property, cf. Lemma 3 in \cite{ap}] \label{straight}
Every paracompact locally definable manifold  $M$  over $R$  has
the following straightening property:

For each set $J \subset [0,1]\times M$ such that the natural projection
$\beta :[0,1]\times M\rightarrow [0,1]$ restricted to $J$ is a covering mapping of finite
degree, there exists an isomorphism, called the \textbf{straightening isomorphism}, $\tau :[0,1]\times M\rightarrow [0,1]\times M$
which satisfies the following three conditions:

2.1) $\beta \circ \tau = \beta,$

2.2) $\tau \mid \{ 0\} \times M = id,$

2.3) $\tau (J) = [0,1]\times (\alpha (J\cap (\{ 0\} \times M))),$ where
 $\alpha :[0,1]\times M\rightarrow M$ is the natural projection.
\end{lem}

\begin{proof}
\textit{Special case.} Assume $M$ is a unit open ball in $R^m$.
The set $J$ is a finite union of  graphs of definable continuous mappings $\gamma_i :[0,1]_R\to M$ $(i=1,...,n)$.
We apply induction on the number $n$ of these graphs.

If $n=1$ then obviously the straightening exists (compare Lemma 2 in \cite{ap}), and the isomorphism may be chosen to extend continuously
to the identity on the unit sphere.

If $n>1$ and the lemma is true for $n-1$, then we can assume that the first $n-1$ graphs (of the functions $\gamma_1,...,\gamma_{n-1}$) are already straightened and that the  distances between images of the corresponding mappings (points $p_1,...,p_{n-1}$) are not infinitesimals. 
Moreover, since the distance from the value $\gamma_n(t)$ of the last function $\gamma_n$ to any of the distinguished points has a positive lower bound,
we can assume $\gamma_n(t)$
is always outside some closed balls centered at $p_i$'s with radius larger than some rational number. Now, we can cover the rest of the unit ball by finitely many regions that are each isomorphic to the open unit ball. 
Since the last function is definable, there is only finitely many transitions from one region to another when $t\in [0,1]_R$. We have the straightening inside each of the regions.
By glueing such straightenings as in the proof of Lemma 3 of \cite{ap}, we get  the straightening of the whole $n$-th mapping.
Again the straightening extends continuously to the identity on the unit sphere.

\textit{General case.}
Again $J$ is a finite union of graphs of definable functions on $[0,1]_R$ (by arguments similar to those of the usual topological context).
Since $J$ is definable, it is contained in a finite union of open sets each isomorphic to the open unit ball in $R^m$. The thesis of the lemma extends by  arguments similar to these of the special case.
\end{proof}

\begin{proof}[Proof of  Theorem \ref{bertini}]
Clearly, $X$ is a connected and locally simply connected space.
Let $ j:L\setminus V\hookrightarrow X\setminus
(\Delta \times Z)$ and $k:X\setminus (\Delta \times Z)\hookrightarrow X$
be the inclusions. Then the proof falls naturally into two parts.

\textit{Step 1.}
{\it The induced mapping $j_{*} :\pi_{1} (L\setminus V)\rightarrow \pi_{1}
(X\setminus (\Delta \times Z))$ is an epimorphism.} This step is analogous to Part 1 of the proof of Theorem 1 in \cite{ap}.
Here Lemma \ref{straight} is used.




\textit{Step 2.}
{\it The mapping $k_{*} :\pi_{1} (X\setminus (\Delta \times Z))\rightarrow
\pi_{1} (X)$ induced by $k$ is an epimorphism.}

 Notice that $(\Delta \times Z)\cap X$ nowhere disconnects $X$.
Let $u=(f,g)$ be a loop in $X$ at $(p,h(p))$. The set $\im(u)$ has an affine open neighborhood $W$.

We use a (locally finite) triangulation ``over $\mb{Q}$'' of $Y\times Z$ (that is an isomorphism $\phi : K\to Y\times Z$ for some strictly locally finite, not necessary closed, simplicial complex $(K,\Sigma(K))$, following the notation of \cite{LSS}), compatible with $\im(u),\Delta\times Z,V,L,h,W$.

There is  $\varepsilon \in \mb{Q}$ such that the ``distance'' from $\phi^{-1}(\im(u))$ to $\phi^{-1}(V\cap W)$ in some ambient affine space is at least $\varepsilon$.  
Moreover, the ``velocity'' of $\phi^{-1}\circ u$ (existing almost everywhere) is bounded from above by some rational number. Since now all the considered sets and functions (appearing in the context of $K$) are piecewise linear over $\mb{Q}$, the Lebesgue number argument is available.
By the use of the ``distance'' function in the ambient affine space and the barycentric coordinates for the chosen triangulation, we find a loop $\tilde{u}=(\tilde{f},\tilde{g})$ homotopic to $u$ 
rel $\{ 0,1\}$ with image in $X\setminus (\Delta \times Z)$.
\end{proof}

 The following facts and theorems, whose proofs use the machinery of
\textit{good triangulations},
 are straightforward generalizations of the corresponding
semialgebraic versions from \cite{LSS}:

\begin{f}[canonical neighborhood retraction, cf.  III.1.1]
Let $M$ be an object of $\rplds(R)$ and $A$ a
 closed subspace. There is an open neighborhood
$U$ (in particular a subspace) of $A$ and a strong deformation retraction
$$H:\overline{U} \times [0,1]\rightarrow \overline{U} $$
from $\overline{U}$ to $A$ such that the restriction $H|U\times [0,1]$
 is a strong deformation retraction from $U$ to $A$.
\end{f}

\begin{f}[extension of morphisms, cf. III.1.2]
Let $M$ be an object of $\rplds(R)$,
 $A$ a closed subspace, and $U$ a neighborhood of $A$
from the previous theorem. Any morphism $f:A\rightarrow Z$ into a
 regular paracompact locally definable space  extends to a morphism
$\tilde{f} :\overline {U} \rightarrow Z$.
Moreover, if $\tilde{f}_1$, $\tilde{f}_2$ are extensions of $f$ to
 $\overline{U}$,
then they are homotopic in $\overline{U}$ relative to $A$.
\end{f}

\begin{f}[homotopy extension property, cf.  III.1.4]
Let $M$ be an object of $\rplds(R)$.
If $A$ is a closed subspace of $M$, then $(A\times [0,1])\cup (M\times \{ 0\}
 )$ is a strong deformation retract of $M\times [0,1]$.
In particular, the pair $(M,A)$ has the following Homotopy Extension Property:

for each morphism $g:M\rightarrow Z$ into a regular
 paracompact locally definable space
$Z$ and a homotopy $F: A\times [0,1] \rightarrow Z$
with $F_0 =g|A$ there exists a homotopy $G:M\times [0,1] \rightarrow Z$
 with $G_0 =g$ and $G|A\times [0,1] =F$.
\end{f}

Since our spaces may be triangulated, the method of \emph{simplicial approximations}
(III.2.5 of \cite{LSS}) makes a good job. In particular, the method of \textit{well cored systems}  and \textit{canonical retractions} from III.2 in \cite{LSS} gives the following

\begin{f} \label{cores}
Each object of $\rplds(R)$ is homotopy equivalent to a partially complete one. A system $(M,A_1,...,A_k)$ of a regular paracompact locally definable space with closed subspaces is homotopy equivalent to an analogous system of partially complete spaces. 
\end{f} 

The following two main theorems from \cite{LSS} generalize,
but the \textit{mapping spaces} from III.3, which depend on the degrees of polynomials, should be replaced with similar mapping spaces depending on concrete formulas 
$\Psi (\overline{x}, \overline{y}, \overline{z})$, with parameters $\overline{z}$, of the language of the structure $R$
(one ``mapping space'' per each formula $\Psi$). 

 \vspace{2mm}
 Let $(M,A_1 ,...,A_r )$ and $(N,B_1 ,...,B_r )$ be systems of regular paracompact locally definable spaces
over $R$, where each $A_i (i=1,...,r)$ is closed in $M$. Let $h: C\rightarrow N $ be a given
morphism from a closed subspace $C$ of $M$ such that $h(C\cap A_i )\subseteq B_i$
for each $i=1,...,r$. Then we have

\begin{thm}[first Comparison Theorem, cf. III.4.2]
Let $R\prec S$ be an elementary extension. 
Then  the ``base field extension'' functor from $R$ to $S$
 induces a bijection between  the homotopy sets:
\begin{displaymath} \kappa :
 [(M,A_1 ,...,A_k ), (N,B_1 ,...,B_k ) ]^{h} \rightarrow
  [( M,A_1 ,...,A_k ), (N,B_1 ,...,B_k ) ]^{h} (S).
\end{displaymath}
\end{thm}

\begin{thm}[second Comparison Theorem, cf.  III.5.1]
Let $R$ be an o-minimal expansion of $\mathbb{R}$.
 Then the ``forgetful'' functor $\rplds(R)\rightarrow \mathbf{Top}$ to the topological category
 induces  a bijection between the homotopy sets
\begin{displaymath} \lambda : [(M,A_1 ,...,A_k ), (N,B_1 ,...,B_k ) ]^{h} \rightarrow
 [(M,A_1 ,...,A_k ), (N,B_1 ,...,B_k ) ]^{h}_{top}.
 \end{displaymath} 
\end{thm}

 Moreover, a version of the proof of the first Comparison Theorem gives

\begin{thm}[third Comparison Theorem]\label{3compld}
If $R'$ is an o-minimal expansion of $R$, then the
``expansion'' functor induces  a bijection between the
homotopy sets
\begin{displaymath} \mu  : 
[(M,A_1 ,...,A_k ), (N,B_1 ,...,B_k ) ]^{h}_{R} \rightarrow
 [(M,A_1 ,...,A_k )_{R'}, (N,B_1 ,...,B_k )_{R'} ]^{h}_{R'}.
 \end{displaymath}
\end{thm}
\begin{proof}[Sketch of proof.]
E. Baro and M. Otero \cite{BaOt} have written a detailed proof of this theorem in the case of systems of definable sets. They use a natural tool of  ``normal triangulations'' from \cite{Ba} to get an applicable version of II.4.3 from \cite{LSS}. The theorem  extends to the general case as in \cite{LSS}.
\end{proof}

Because of the locally finite character of the regular paracompact locally definable spaces,
by inspection of the proof of the triangulation theorem (II.4.4 in \cite{LSS}),
each such space has an isomorphic copy that is built from sets definable without parameters
 glued together along sets that are definable without paramaters.
 It is possible to triangulate even ``over the field of real algebraic numbers $\overline{\mb{Q}}$ '' or ``over the field of rational numbers $\mb{Q}$''.
Moreover, if two 0-definable subsets of $R^n$
are isomorphic as definable spaces (i. e. definably homeomorphic), then there is a 0-definable isomorphism  between them
(we may change arbitrary parameters into 0-definable parameters in the defining formula  of an isomorphism).

By the \textit{(noncompact) o-minimal version of Hauptvermutung} for the structure $R$, we understand the following statement, which is a version of Question 1.3 in \cite{BO}:

\textit{Given two semialgebraic (definable in the field structure of $R$) sets in some $R^n$, if they are definably homeomorphic, then they are semialgebraically homeomorphic.}  

In other words:\textit{ if two affine semialgebraic spaces are isomorphic as definable spaces, then they are  isomorphic as semialgebraic spaces.}


It follows from Theorem 2.5 in \cite{Shiota} that this statement is true for every $R$.
Thus the category of regular paracompact locally semialgebraic spaces 
$\rplss(R)$ over (the underlying field of) $R$ may be viewed as a subcategory of $\rplds(R)$, but not as a full subcategory. Moreover, by triangulation with vertices having coordinates in the field of real algebraic numbers $\overline{\mathbb{Q}}$, we have the following fact:

 \textit{Each regular paracompact locally definable space over $R$ is \emph{isomorphic} to a regular paracompact locally semialgebraic space over (the underlying field of) $R$.}

Thus,  by the third Comparison Theorem, the homotopy categories $H\rplss(R)$ and $H\rplds(R)$ are equivalent.
Analogously, we get

\begin{cor} \label{concld}
The homotopy categories  of: systems $(M,A_1,...,A_k)$ of regular paracompact locally definable spaces with finitely many closed subspaces and  systems $(M,A_1,...,A_k)$ of regular paracompact locally semialgebraic spaces with finitely many closed subspaces 
(over the ``same'' $R$) are equivalent.
\end{cor}
\begin{proof}
By the triangulation theorem \ref{triang}, every object of the former category is isomorphic to an object of the later category. Thus the ``expansion'' functor is essentially surjective.
By the Comparison Theorem \ref{3compld}, it is also full and faithfull. This implies that this functor is an equivalence of categories. 
\end{proof}

It follows that the homotopy theory for regular paracompact locally definable spaces  can 
to a large extent
be transfered from the semialgebraic homotopy theory and, eventually, from the topological homotopy theory, as in \cite{LSS}.

Other important facts about regular paracompact locally definable spaces will
be developed in a more general setting of definable CW-complexes and weakly definable spaces.

\section{Weakly definable spaces}

In homotopy theory one needs to use quotient spaces
(e. g. mapping cylinders, mapping cones, cofibers,  smash products, reduced suspensions, CW-complexes),
and this operation is not always executable in the category of
locally definable spaces (as in the semialgebraic case). That is why  weakly definable spaces, which are analogues of arbitrary Hausdorff topological spaces, need to be introduced.
We start here to re-develop the theory of M. Knebusch from \cite{WSS}.

Let $(M, O_M )$ be a space over $R$, and let $K$ be a small subset of $M$.
We can induce a space on $K$ in the following way:

i) \textit{open} sets in $K$ are the intersections of open sets on $M$ with $K$,

ii) \textit{admissible} coverings in  $K$ are such open coverings that some
 finite subcovering
already covers the union,

iii) a function $h:V\rightarrow R$ is a \textit{section} of $O_K (V)$
 if it is a finite open  union  of restrictions to $K$ of sections of the sheaf
$O_M$.\\
We call $(K, O_K )$ a \textbf{small subspace} of $(M, O_M )$.

A subset $K$ of $M$ is called \textbf{closed definable} in $M$ if $K$ is
 closed, small, and the space
$(K, O_K)$ is a definable space. The collection of closed definable subsets of $M$ is denoted by $\overline{\gamma}(M)$.
The set $K$ is called a \textbf{polytope} if it is a closed definable complete space.
  We denote the collection of polytopes of $M$ by $\gamma_c (M)$.

A \textbf{weakly definable space} (over $R$) is a space $M$ (over $R$)
 having a family, indexed by a partially ordered set $A$,
of regular closed   definable subsets $(M_{\alpha} )_{\alpha \in A}$ such that the following conditions hold:

WD1) $M$ is the union of all $M_{\alpha}$,

WD2) if $\alpha \le \beta$ then $M_{\alpha}$ is a (closed) subspace of $M_{\beta}$,

WD3) for each $\alpha$ there is only a \underline{finite} number of $\beta$ such that $\beta \le \alpha$,

WD4) the family $(M_{\alpha} )$ is strongly inverse directed, i. e.
 for each $\alpha, \beta$ there is some
$\gamma$ such that $\gamma \le \alpha$, $\gamma \le \beta$ and $M_{\gamma}
 = M_{\alpha} \cap M_{\beta} $,

WD5) the set of indices is directed: for each $\alpha, \beta$ there is $\gamma$ with $\gamma \ge \alpha$, $\gamma \ge \beta$,

WD6) the space $M$ is the \textit{inductive limit} of the spaces $(M_{\alpha} )$,
 what means the following:

a) a subset $U$ of $M$ is \textit{open} iff each $U\cap M_{\alpha}$ is open
 in $M_{\alpha}$,

b) an open  family $(U_{\lambda} )$ is \textit{admissible} iff for each $\alpha$ the
  restricted family
$(M_{\alpha} \cap U_{\lambda} )$ is admissible in $M_{\alpha}$,

c) a function $h:U\rightarrow R$ on some open $U$ is a \textit{section} of $O_M$
 iff all the restrictions
$h|U\cap M_{\alpha}$ are sections of respective sheaves $O_{M_{\alpha}}$.

Such family $(M_{\alpha} )$ is called an \textbf{exhaustion} of $M$.

A space $M$ is called a \textbf{weak polytope} if $M$ has an exhaustion
 composed of polytopes. \textbf{Morphisms} and \textbf{isomorphisms} of weakly definable spaces are their morphisms and isomorphisms as spaces (we get the full subcategory $\wds(R)$
of $\spa(R)$). 
 
A \textbf{weakly definable subset} is such a subset $X\subseteq M$
that: has definable intersections with all members of some exhaustion $(M_{\alpha} )$,
and is considered with the exhaustion $(X\cap M_{\alpha} )$, hence it may be considered as a \textbf{subspace} of $M$ (cf. IV.3 in \cite{WSS}). 
 
A subset $X$ of  $M$ is \textbf{definable} if it is weakly definable and the space $(X,O_X)$ is definable. A subset $X$ of $M$ is definable iff it is weakly definable and is contained in a member of an exhaustion $M_{\alpha}$ (cf. IV.3.4 of \cite{WSS}).

The \textbf{strong topology} on $M$ is the topology that makes the topological space $M$ the respective inductive limit of the topological spaces $M_{\alpha}$.
The unpleasant fact about the weakly definable spaces (comparising with the locally definable spaces)  is that points may not have small neighborhoods (see Example \ref{26}).
Moreover, the open sets from the generalized topology may not form  a basis of the strong topology (cf. Appendix C in \cite{WSS}).

The closure of a definable subset of $M$ is always definable (cf. IV.3.6 of \cite{WSS}), so the topological closure operator restricted to the class $\gamma(M)$ of definable subsets of $M$ may be treated as the closure operator of the generalized topology. 
The weakly definable subsets are  ``piecewise constructible'' from the generalized topology (compare \cite{ap2}).

All  weakly definable spaces are Hausdorff, actually even ``normal'', see IV.3.12 in \cite{WSS}.  We can consider ``expansions'' and ``base field extensions'' of weakly definable spaces (compare considerations in IV.2) or morphisms (in the case of a base field extension) 
similar to the operations defined for locally definable spaces. They do not depend on the chosen exhaustion and preserve connectedness (cf. IV.2 and IV.3 of \cite{WSS}).

\begin{rem}
Assume that a weakly definable space $M$ is also locally definable.
Then $\overline{\gamma}(M)$ is the family of all closed small subsets (as in \cite{LSS}, p. 57), since closed small subsets are definable as subspaces.
We can speak about complete subspaces of $M$. It is easy to see that complete subspaces are always closed. Thus the family $\gamma_c (M)$ contains exactly the definable complete subspaces (as in \cite{LSS}, p. 81). 
\end{rem}

Fiber products exist in $\wds(R)$ (cf. IV.3.20 of \cite{WSS}).
So we (analogously to the case of locally definable spaces) define \textbf{proper} and \textbf{partially proper} mappings between weakly definable spaces as well as \textbf{complete} and \textbf{partially complete}
spaces. It appears that the complete spaces are the polytopes, and the partially complete spaces are the weak polytopes (cf. IV.5 in \cite{WSS}).

The following examples from \cite{WSS} remain relevant in the case of an  o-minimal expansion of a real closed field.

\begin{exam}[cf. IV.1.5 in \cite{WSS}]
The category $\rplds(R)$ is a full subcategory of $\wds(R)$. An exhaustion of an object $M$ of $\rplds(R)$ is given by all finite subcomplexes $Y$ in $X$ that are closed in $X$ for some triangulation $\phi :X\to M$.
\end{exam}

\begin{exams}[cf. IV.1.8 and IV.4.7-8 in \cite{WSS}] \label{26}
 An infinite wedge of circles is a weak polytope but not
 a locally definable space.
 A ``countable comb'' or ``uncountable comb'' is a weak polytope which is not a locally  definable space.
\end{exams}

\begin{ew}[cf. IV.4.7 in \cite{WSS}] 
Consider the ``countable comb'' from IV.4.7 in \cite{WSS}. This example shows that the topological closure of a weakly definable subset may not be weakly definable. Moreover,  the naive ``Arc Sellecting Lemma for weakly definable spaces'' does not hold. 
\end{ew}

On the other hand, the following examples did not appear explicitely
in \cite{WSS}.

\begin{exam}
Consider an uncountable proper subfield $F$ of $\mb{R}$. Let $X$ be a subset of the unit square
$[0,1]^2$ consisting of points that have at least one coordinate in $F$.
This set has a natural exhaustion making $X$ into a weak polytope over $\mb{R}$. This weak polytope is not locally simply connected.
\end{exam}

\begin{exam} An open interval of $R$ is a definable space but not a weak polytope, an infinite comb with such a ``hand'' is a weakly definable space but not a weak polytope.
\end{exam}

Glueing weakly definable spaces is possible: for a closed pair $(M,A)$ and a partially proper morphism $f:A\rightarrow N$ the quotient space of $M \sqcup N$
by an equivalence relation identifying each $a\in A$ with $f(a)$ is a weakly definable space $M\cup_{f} N$ called the \textbf{space obtained by glueing $M$ to $N$ along $A$ by $f$}. Then the projection $\pi :M\sqcup N \to M\cup_{f} N$ is partially proper and strongly surjective, cf. IV.8.6
in \cite{WSS}. (A morphism $f:M\rightarrow N$ is \textbf{strongly surjective} if each definable subset of $N$ is covered by the image of a definable subset of $M$.)

A family $\mc{A}$ of subsets of  a weakly definable space $M$ will be called \textbf{piecewise finite} if for each $D\in \gamma(M)$, the set $D$ meets only finitely many members of $\mc{A}$. 
(Such families are called ``partially finite'' in \cite{WSS}.)

A \textbf{definable partition} of a weakly definable space $M$ is a piecewise finite partition
of $M$ into a subset $\Sigma$ of the family  $\gamma (M)$ of definable subsets of $M$.
An element $\tau$ of $\Sigma$ is an \textbf{immediate face} of $\sigma$
if $\tau \cap (\overline{\sigma} \setminus \sigma )\neq\emptyset$.
Then we write $\tau \prec \sigma $. A \textbf{face} of $\sigma$
is an element of some finite chain of immediate faces finishing with $\sigma$.
 (Each $\sigma$ has only finitely many immediate faces, even a finite numer of faces, cf. V.1.7 in \cite{WSS}).

A \textbf{patch decomposition} of $M$ is a definable partition $\Sigma$ of $M$ such that:
 for each $\sigma \in \Sigma$ there is a number $n \in \N$ such that any  chain
$\tau_r \prec \tau_{r-1} \prec ... \prec \tau_0 =\sigma$ in $\Sigma$ has length
$r\le n$.
The smallest such $n$ is called the \textbf{height} of $\sigma$ and denoted by $h(\sigma )$.
A \textbf{patch complex} is a pair $(M, \Sigma (M))$ consisting of a space $M$ and
a patch decomposition $\Sigma (M)$ of $M$. Elements of the patch decomposition are called \textbf{patches}. 

\begin{exam}[cf. V.1.4]
Each exhaustion gives a patch decomposition of $M$.
\end{exam}

Instead of triangulations for $\rplds(R)$, we have available for $\wds(R)$
so called special patch decompositions. A \textbf{special patch decomposition}
is such a patch decomposition that for each $\sigma \in \Sigma$, the pair $(\overline{\sigma},\sigma)$ is isomorphic to the pair with the second element being a standard open simplex, and the first element this standard open simplex with some added open proper faces.

\begin{f}[cf. V.1.12 in \cite{WSS}]  
Let $M$ be an object of $\wds(R)$ and let $\mc{A}$ be a piecewise finite family of subspaces.
Then there is a simultaneous special patch decomposition of $M$ and the family $\mc{A}$.
\end{f}

A \textbf{relative patch decomposition} of a closed pair $(M,A)$ is
a patch decomposition $\Sigma$ of the space $M\setminus A$.
Then we denote by $\Sigma (n)$  the union of all patches of height $n$,
by $M_n$ the union of $A$ and all $\Sigma (m)$ with $m\le n$,
 $M(n)$ the ``direct (generalized) topological sum'' of all closures $\overline{\sigma}$ where $\sigma \in \Sigma(n)$, and $\partial M(n)$ the direct sum of all frontiers $\partial \sigma
 =\overline{\sigma} \setminus \sigma $ of $\sigma \in \Sigma (n)$.
 
By $\psi_n :M(n) \rightarrow M_n$ we denote the union of all inclusions
$\overline{\sigma} \rightarrow M_n$ with $\sigma \in \Sigma (n)$, and by $\phi_n : \partial M(n) \rightarrow M_{n-1}$ the restriction of $\psi_n$, which is called the \textbf{attaching map}.     
Then, since $\phi_n $ is partially proper (cf. VI.2  in \cite{WSS}), we can express $M_n$ as $M(n) \cup_{\phi_n} M_{n-1}$. The space $M_n$ is called
\textbf{n-chunk} and $M(n)$ is called \textbf{n-belt}. So each weakly definable space 
is built up by glueing direct (generalized) topological sums of definable spaces to the earlier constructed spaces in countably many steps.
In particular, definable versions of CW-complexes are among weakly definable spaces
(see below).

A family $(X_{\lambda} )_{\lambda \in \Lambda}$ from $\mathcal{T} (M)$, the class of weakly definable subsets of $M$,  is called \textbf{admissible} if each definable subspace $B$ of $M$ is contained in the union of finitely many elements of the family. (One could call such families ``piecewise essentially finite'' or ``partially essentially finite''.)
Thus definable partitions are exactly the admissible partitions into definable subsets.
 
An \textbf{admissible filtration} of a space $X$ is an admissible increasing sequence of closed subspaces $(X_n )_{n\in \N}$  covering $X$. For example: the sequence $(M_n )_{ n\in \N}$ of chunks of $M$ (for a given patch decomposition) is an admissible filtration of $M$ (cf. VI.2 in \cite{WSS}). 

The next fact is very important in homotopy-theoretic considerations.

\begin{f}[composition of homotopies, cf. V.5.1]\label{comphom}
Let $(C_n)_{n\in \N}$ be an admissible filtration of a space $M$.
Assume $(G_n :M\times [0,1]\rightarrow N)_{n\in \N}$ is a family of homotopies
such that $G_{n+1} (\cdot ,0)=G_{n} (\cdot ,1)$ and $G_{n} $ is constant on $C_n$ .
For any given strictly increasing sequence $0=s_0 <s_1 <s_2 <...$ with all $s_m$ less than 1 there is a homotopy $F:M\times [0,1]\rightarrow N$ such that 
$$F(x,t) = G_{k+1} (x,\frac{t - s_k}{s_{k+1} - s_k}), \mbox{ for }(x,t)\in C_n\times [s_k ,s_{k+1} ],
 0\le k\le n-2,$$ 
  and  $F(x,t)=G_n (x,0)$ for $(x,t)\in C_n \times [s_{n-1},1] $.
\end{f}

\section{Comparison Theorems for weakly definable spaces}

Now, with {patch decompositions} playing the role of triangulations
we get the Comparison Theorems for weakly definable spaces as in \cite{WSS}.

\begin{f}[homotopy extension property, cf. V.2.9]\label{hep}
Let $(M,A)$ be a closed pair of weakly definable spaces over $R$. 
Then $(A\times [0,1])\cup (M\times \{ 0\} )$ 
is a strong deformation retract of $M\times [0,1]$.
In particular, the pair $(M,A)$ has the following Homotopy Extension Property:

for each morphism $g:M\rightarrow Z$ into a weakly definable space
$Z$ and a homotopy $F: A\times [0,1] \rightarrow Z$
with $F_0 =g|A$ there exists a homotopy $G:M\times [0,1] \rightarrow Z$
 with $G_0 =g$ and $G|A\times [0,1] =F$.
\end{f}

\vspace{2mm}
Let $(M,A_1 ,...,A_r )$ and $(N,B_1 ,...,B_r )$ be systems of weakly definable spaces
over $R$ where each $A_i$ is closed in $M$. Let $h: C\rightarrow N $ be a given
morphism from a closed subspace $C$ of $M$ such that $h(C\cap A_i )\subseteq B_i$
for each $i=1,...,r$. Then we have

\begin{thm}[first Comparison Theorem, cf. V.5.2 i)]
 For an elementary extension $R\prec S$ the following map, induced by the ``base field extension'' functor, is a bijection
$$\kappa : [(M,A_1 ,...,A_r ),(N,B_1 ,...,B_r )]^{h} \rightarrow
[(M,A_1 ,...,A_r ),(N,B_1 ,...,B_r )]^{h} (S) .$$
\end{thm}

\begin{thm}[second Comparison Theorem, cf. V.5.2 ii)]
 If $R=\R$ as fields, then the following map  to the topological homotopy sets, induced by the ``forgetful'' functor, is a bijection
$$\lambda :[(M,A_1 ,...,A_r ),(N,B_1 ,...,B_r )]^{h} \rightarrow
[(M,A_1 ,...,A_r ),(N,B_1 ,...,B_r )]^{h}_{top} .$$
\end{thm}
Again, a version of the proof of the first Comparison Theorem (thus a version of the proof of V.5.2 i); we present the proof for the convenience of the reader) gives:
\begin{thm}[third Comparison Theorem]\label{3compwd}
If $R'$ is an o-minimal expansion of $R$, then the following map, induced by the 
``expansion'' functor,  is a bijection
\begin{displaymath} \mu : [(M,A_1 ,...,A_k ), (N,B_1 ,...,B_k ) ]^{h}_{R} \rightarrow
 [(M,A_1 ,...,A_k )_{R'}, (N,B_1 ,...,B_k )_{R'} ]^{h}_{R'}.
 \end{displaymath}
\end{thm}
\begin{proof}
It suffices to prove the surjectivity, and only the case $k=0$. We have a map $f:M\to N$ 
(over $R'$) extending $h:C\to N$ (over $R$), and we seek for a mapping $g:M\to N$  (over $R$)
such that $g$ is homotopic to $f$ relative to $C$ (the homotopies appearing in this proof are allowed to be over $R'$).

We choose a relative patch decomposition (over $R$) of $(M,C)$, and will construct maps $h_n:M_n\to N$ (over $R$), $f_n:M\to N$ (over $R'$) for $n\geq -1$, and a homotopy $H_n:M\times [0,1]\to N$ relative $M_{n-1}$ such that: $h_{-1}=h$, $h_n|_{M_{n-1}}=h_{n-1}$, $f_{-1}=f$, $f_n|_{M_n}=h_n$, $H_n(\cdot,0)=f_{n-1}$, $H_n(\cdot,1)=f_n$. 
If we do this, we are done:
we have a map $g:M\to N$ with $g|_{M_n}=h_n$ for each $n$. Composing, by Fact \ref{comphom}, the homotopies $(H_n)_{n\geq 0}$ along a sequence $s_n\in [0,1)$ with $s_{-1}=0$, we obtain a homotopy $G:M\times [0,1]\to N$ relative $C$ from $f$ to $g$ as desired.

We start with $h_{-1}=h$, $f_{-1}=f$. Assume that $h_i,f_i,H_i$ are given for $i<n$.
Then we get a pushout diagram over $R$ (see page 149 of \cite{WSS}) and we define:
$$ k_n =h_{n-1}\circ \phi_n :\partial M(n)\to N \mbox{ (over $R$)},$$
$$ u_n= (f_{n-1}|_{M_n})\circ \psi_n :M(n)\to N \mbox{ (over $R'$)}.$$
Notice that $u_n$ extends $k_n$. By the Comparison Theorem for locally definable spaces (Theorem \ref{3compld})
there is a map $v_n:M(n)\to N$  over $R$ extending $k_n$ and a homotopy $F_n:M(n)\times [0,1]\to N$ relative $\partial M(n)$ from $u_n$ to $v_n$. The maps $v_n$ and $h_{n-1}$ combine to a map $h_n:M_n\to N$, with $h_n\circ\psi_n=v_n$ and $h|_{M_{n-1}}=h_{n-1}$.
The map $F_n$ and $M_{n-1}\times [0,1]\ni (x,t)\mapsto h_{n-1}(x)\in N$ combine 
(cf. IV.8.7.ii) in \cite{WSS}) to the homotopy $\tilde{H}_n:M_n\times [0,1]\to N$
relative $M_{n-1}$ from $f_{n-1}|_{M_n}$ to $h_n$. It can be extended (by Fact \ref{hep}) 
to the homotopy $H_n:M\times [0,1]\to N$ with $H_n(\cdot,0)=f_{n-1}$. Put $f_n=H_n(\cdot,1)$.
This finishes the induction step and the proof of the theorem.
\end{proof}

Again, the category of weakly semialgebraic spaces over (the underlying field of) $R$ may be considered  a (not full in general) subcategory of  $\wds(R)$. But see the following important new example:
 
\begin{ew} 
Let $Q$ be the square $[0,1]^2_{R}$. Now form $\widetilde{Q}$ in the following way: for each definable subset $A$ of $Q$  glue $A\times S^1$ to $Q$ by identifying $A\times \{ 1\} $ with $A$.
If there are definable non-semialgebraic sets in $R^2$, then $\widetilde{Q}$ as a weakly definable space is not isomorphic to (an expansion of) a weakly semialgebraic space over $R$.
\end{ew}

\section{Definable CW-complexes}

A \textbf{relative definable CW-complex} $(M,A)$ over $R$  is a relative patch complex $(M,A)$ satisfying the conditions:

(CW1) immediate faces of patches have smaller dimensions than the original patches
in the patch decomposition of $M\setminus A$,

(CW2) for each patch $\sigma \in \Sigma (M,A)$ there is a morphism $\chi_{\sigma} : E_n \rightarrow \overline{\sigma}$ ($E_n$ denotes the unit closed ball of dimension $n$) that maps the open ball isomorphically onto $\sigma$ and the sphere
onto $\partial  \sigma$.
For $A=\emptyset$, we have an \textbf{absolute definable CW-complex} over $R$. 
All definable CW-complexes are weak polytopes (absolute or relative, see V.7, p. 165, in \cite{WSS}). A \textbf{system of definable CW-complexes} is a system of spaces $(M,A_1,...,A_k)$ such that each $A_i$ is a closed subcomplex of the definable CW-complex $M$ (cf. V.7, p. 178, of \cite{WSS}). Such a system is \textbf{decreasing} if $A_i$ is a (closed) subcomplex of $A_{i-1}$ for $i=1,...,k$, where $A_0=M$.
As in the semialgebraic  case, we have the following.

\begin{exam} \label{fact}
Each partially complete object of  $\rplds(R)$ admits a definable CW-complex structure over $R$, since it is isomorphic to a closed (geometric) locally finite simplicial complex. (Compare considerations of II.4 and ii) in Examples V.7.1.)
\end{exam}

 Fact \ref{cores} and Example \ref{fact} give
 
\begin{f} \label{locdefred}
Each object of  $\rplds(R)$ is homotopy equivalent to a definable CW-complex over $R$. Each system $(M,A_1,...,A_k)$ of a regular paracompact locally definable space with closed subspaces is homotopy equivalent to a system of definable CW-complexes.
\end{f}

The following version of the Whitehead theorem for definable CW-complexes may be proved like 
its topological analogue (see Theorem 7.5.4 in \cite{maunder}).

\begin{thm}\label{cwwhitehead}
Each weak homotopy equivalence between definable CW-complexes is a homotopy equivalence.
Similar facts hold for any decreasing  systems of definable CW-complexes.
\end{thm}
\begin{proof}
The proof is analogous to the proofs of 7.5.2, 7.5.3 and 7.5.4 in \cite{maunder}.
The argument from the long exact homotopy sequence may be proved like in \cite{Hu} (compare
III.6.1 in \cite{LSS} and V.6.6 in \cite{WSS}). The second part of 
the thesis follows from the definable analogue of V.2.13 in \cite{WSS}.
\end{proof}

Using the above instead of Theorem V.6.10 of \cite{WSS}, we can both 
pass to a reduct and eliminate parameters. 

\begin{thm}[cf. V.7.10 in \cite{WSS}]\label{elimination}
Each  definable CW-complex is homotopy equivalent to an expansion of a base field extension of a semialgebraic CW-complex over $\overline{\mb{Q}}$. Analogous facts hold for decreasing systems of definable CW-complexes.
\end{thm}
\begin{proof}
This follows from the reasoning with relative CW-complexes analogous to the proof of V.7.10 in \cite{WSS} (instead of the case of an elementary extension of real closed fields, we have the case of an o-minimal expansion of a real closed field).
The construction of the desired relative CW-complex ``skeleton by skeleton'' is similar.
Since we are dealing only with decreasing systems of definable CW-complexes, the use of V.6.10 of \cite{WSS} (whose role is the transition from finite unions to any unions)  may be replaced with the use of Theorem \ref{cwwhitehead}. 
\end{proof}

Moreover, combining the above with the Comparison Theorems gives
an extension of Remarks VI.1.3 of \cite{WSS}. 

\begin{cor}\label{equi}
The homotopy categories of: topological CW-complexes, semialgebraic CW-complexes over (the underlying field of) $R$, and definable CW-complexes over $R$ are equivalent.
Similar facts hold for decreasing systems of CW-complexes.
\end{cor}

\section{The case of  bounded o-minimal theories}

Let $T$ be an  o-minimal complete theory extending RCF. We may assume that the theory is already Skolemized, so every 0-definable function is in the language and  $T$ has quantifier elimination. We can build models of $T$ using the \textit{definable closure} operation in some huge model (or, equivalently, using the notion of a \textit{generated substructure} of a huge model for the  chosen rich language).  
Taking a ``primitive extension''  generated by a single element $t$ over a model $R$ gives a new model $R\langle t\rangle$ of $T$ determined up to isomorphism by the type this single element realizes over the former model
$R$.  

Such a $T$ will be called \textbf{bounded} if the model $P\langle t\rangle$ has countable
cofinality, where $P$ is the prime model of $T$ and $t$ realizes $+\infty$ over $P$. 
This condition can be expressed in the following words: there is a (countable) sequence of 0-definable unary functions that is cofinal in the set of all 0-definable unary functions at $+\infty$ (this property does not depend on a model of $T$). In particular,
polynomially bounded theories are bounded.
Notice that $P\langle t\rangle$ is cofinal in $R\langle t\rangle$, for any model $R$ of $T$, if $t$ realizes $+\infty$ over $R$. 
 
Each bounded theory $T$ has the following property:  each model $R$ has an elementary extension $S$ such that both $S$ and its ``primitive extension''
$S\langle t\rangle$, with $t$ realizing $+\infty $ over $S$, have countable cofinality. (Take $S=R\langle t_1 \rangle$, with $t_1$ realizing $+\infty$ over $R$). This allows, by the first Comparison Theorem, to extend many facts about weakly definable spaces over ``nice'' models to spaces over any model of $T$.

The following example may be extracted from the proof of Theorem IV.9.2 in \cite{WSS}. It shows the importance of the boundedness assumption.
(The role of the boundedness assumption may be also seen by considering Example IV.9.12 in \cite{WSS}.)

\begin{exam}
Consider the closed $m$-dimensional simplex with one open proper face removed ($m\geq 2$), call this set $A$, as a definable subset of $R^{m+1}$.
 We want to introduce a partially complete space on the same set $A$.
If $R$ and $R\langle t\rangle$ have countable cofinality, then we can find a sequence of internal points tending to the barycenter of the removed face, and we can use a ``cofinal at $0_{+}$'' sequence of unary functions tending (even uniformly) to the zero function to produce an increasing sequence $(P_n)_{n\in \mb{N}}$ of polytopes covering our set $A$ and such that any polytope contained in $A$ is contained in some $P_n$. Then $(P_n)_{n\in \mb{N}}$ is an exhaustion of a weak polytope with the underlying set $A$. The old space and the new space on $A$ have the same polytopes.(Compare the proof of Theorem IV.9.2.) A similar construction can be made if several open proper faces are removed.
\end{exam}

By the reasoning similar to that of V.7.8, we get

\begin{thm}[CW-approximation, cf. V.7.14] \label{CWappr} If $T$ is bounded, then each decreasing system of weakly definable spaces $(M_0 ,...,M_r )$ over $R$ has a  CW-approximation  (that is a morphism $\phi : (P_0 ,...,P_r ) \rightarrow (M_0 ,...,M_r )$ from a decreasing system of definable CW-complexes over $R$ that is a homotopy equivalence of systems of spaces).
\end{thm}

The methods to obtain this theorem include the use (as in  IV.9-10 of \cite{WSS}) of a so called \textbf{partially complete core} $P(M)$ \textbf{of a weakly definable space} $M$, which is an analogue and generalization of the localization $M_{loc}$ for locally complete paracompact
locally definable spaces $M$, and  a \textbf{partially proper core} $p_{f}$ \textbf{of  a morphism} $f:M\rightarrow N$ of weakly definable spaces. 
(Note that it is sensible to ask for a partially complete core only if $R$ has countable cofinality.)
In particular,  the Strong Whitehead Theorem (cf. V.6.10), proved by methods of  IV.9-10 and V.4.7, V.4.13 in \cite{WSS}, guarantees the extension of relevant results to weakly definable spaces.
Thus the homotopy category of decreasing systems of weakly definable spaces over $R$ is
equivalent to its full (homotopy) subcategory of decreasing systems of  definable CW-complexes over $R$ (one uses an analogue of Theorem V.2.13 in \cite{WSS}).

The following corollary is an extension of Corollary \ref{equi} in the bounded case.

\begin{cor}\label{concwd}
If $T$ is bounded, then the homotopy categories of 
weakly definable spaces (over any model $R$ of $T$) and of
topological, semialgebraic and definable CW-complexes are  all equivalent. 
Similarly for decreasing systems  of spaces.
\end{cor}

Still the homotopy category of $\wds(R)$ may possibly be reacher in the non-bounded case.  

\section{Generalized homology and cohomology theories} 

Now we have the  operation of taking the (reduced) suspension $SM=S^1 \wedge M$ on the category of pointed weak polytopes $\mathcal{P}^{*}(R)$ over $R$, and on its homotopy category $H\mathcal{P}^{*}(R)$ (cf. VI.1 in \cite{WSS}).
This allows to define analogues of so called \textit{complete generalized homology
 and cohomology theories}, known from the usual homotopy theory, 
just as in VI.2 and VI.3 of \cite{WSS}.
(Such theories do not necessarily satisfy the \textit{dimension axiom}.) 
Denote the category of abelian groups by $Ab$.
 For a pair $(M,A)$ of pointed weak polytopes,  $M/A$ will denote the quotient space of $M$ by a closed space $A$, with the distinguished point being the point obtained from $A$.

A \textbf{reduced cohomology theory} $k^{*}$ over $R$ is a sequence 
$(k^n)_{n\in \mathbb{Z}}$ of contravariant functors
$k^n :H\mathcal{P}^{*} (R)\rightarrow Ab$ together with natural equivalences $s^n : k^{n+1}\circ S \leftrightsquigarrow k^n$ such that the following hold:

\textbf{Exactness axiom}

For each $n\in \mathbb{Z}$ and each pair of pointed weak polytopes $(M,A)$
the sequence
$$ k^n (M/A) \stackrel{p^{*}}{\rightarrow } k^n (M) \stackrel{i^{*}}{\rightarrow}
k^n (A) $$
is exact.

\textbf{Wedge Axiom}

For each $n\in \mathbb{Z}$ and each family $(M_{\lambda})_{\lambda \in  \Lambda}$
of pointed weak polytopes the mapping
$$ (i_{\lambda} )^{*} : k^n (\bigvee_{\lambda} M_{\lambda}) \rightarrow 
\prod_{\lambda}  k^n (M)     $$
is an isomorphism.

A \textbf{reduced homology theory} $h_*$ over $R$ is a sequence $(h_n)_{n\in \mathbb{Z}}$ of covariant functors
$h_n :H\mathcal{P}^* (R)\rightarrow Ab$ together with natural equivalences $s_n : h_n \leftrightsquigarrow h_{n+1} \circ S$ such that the following hold:

\textbf{Exactness axiom}

For each $n\in \mathbb{Z}$ and each pair of pointed weak polytopes $(M,A)$
the sequence
$$ h_n (A) \stackrel{i_{*}}{\rightarrow } h_n (M) \stackrel{p_{*}}{\rightarrow}
h_n (M/A) $$
is exact.

\textbf{Wedge Axiom}

For each $n\in \mathbb{Z}$ and each family $(M_{\lambda})_{\lambda \in\Lambda}$
of pointed weak polytopes the mapping
$$ (i_{\lambda} )_{*} : \bigoplus_{\lambda} h_n (M_{\lambda} )\rightarrow h_n (\bigvee_{\lambda} M_{\lambda})  $$
is an isomorphism.

\vspace{2mm}
If $T$ is bounded, then  these theories correspond uniquely (up to an isomorphism) to topological theories (cf.  VI.2.12 and VI.3  in \cite{WSS}).
All these generalized homology and cohomology functors
can be built by using spectra for homology theories, or $\Omega$-spectra for cohomology theories as in  VI.8 of \cite{WSS}.

Similarly, \textit{unreduced} generalized \textit{homology} and \textit{cohomology} theories may be considered on the category $H\mc{P}(2,R)$ of pairs of weak polytopes. 
If $T$ is bounded, then these theories are equivalent to respective reduced theories, cf. VI.4 in \cite{WSS}; homology theories are extendable to $H\wds(2,R)$, and some difficulties appear for cohomology theories, cf. VI.5-6 in \cite{WSS}.  We get the following extension of Corollaries \ref{equi} and \ref{concwd}.
 
\begin{cor}\label{conc}
If $T$ is bounded, then, by the equivalence of respective homotopy categories of  topological pointed CW-complexes (with continuous mappings) and of pointed weak polytopes, we get ``the same'' generalized homology and cohomology theories as the classical ones, known from the usual topological homotopy theory.
\end{cor}

\section{Open problems}
The following problems are still open:

1) Can the assumption of boundedness of $T$ in Theorem \ref{CWappr} and later be omitted?
Is there a way of proving the Strong Whitehead Theorem (the analogue of V.6.10) without methods of IV.9-10 
of \cite{WSS}?


2) Do the above consideratons lead to a ``(closed) model category'' (see \cite{h}, page 109, for the definition)? Such categories are desired in (abstract) homotopy theory.

\vspace{5mm}
\textbf{Acknowledgements.}
I acknowledge support from the european Research Training Network RAAG (HPNR-CT-2001-00271)
that gave me the opportunity to speak about locally and weakly definable spaces in Perugia (Italy) in 2004 and in Passau (Germany) in 2005. The organizers of the semester
``Model Theory and Applications to Algebra and Analysis''  
held in the Isaac Newton Institute in Cambridge (UK) gave me the opportunity to present a poster during the workshop ``An Introduction to Recent Applications of
Model Theory'', 29 March -- 8 April 2005, supported by the European Commission (MSCF-CT-2003-503674).
I thank Alessandro Berarducci, Margarita Otero, and Kobi Peterzil
for pointing out some mistakes in the previous version of this paper during the 
workshop ``Around o-minimality'' organized by Anand Pillay, held in Leeds, in March 2006.
Some part of the work on this paper was done during my stay at the Fields Institute
during the Thematic Program on o-minimal Structures and Real Analytic Geometry in 2009.
I would like to thank the Fields Institute (primarily the Organizers of the semester) for warm hospitality.

{\sc Politechnika Krakowska

Instytut Matematyki

Warszawska 24

PL-31-155 Krak\'ow

Poland}

E-mail: \textit{pupiekos@cyf-kr.edu.pl}
\end{document}